\documentclass[12pt,a4paper]{article}
\usepackage{amsmath,amssymb,amsthm}
\usepackage{hyperref}
\usepackage{enumitem}
\usepackage{tikz-cd}
\usepackage{fullpage}

\def\beq{\begin{equation}}
\def\eeq{\end{equation}}

\numberwithin{equation}{section}

\newtheorem{thm}{Theorem}[section]
\newtheorem{lemma}[thm]{Lemma}
\newtheorem{prop}[thm]{Proposition}
\newtheorem{cor}[thm]{Corollary}

\theoremstyle{definition}
\newtheorem{defn}[thm]{Definition}

\theoremstyle{remark}
\newtheorem{rmk}[thm]{Remark}

\def\Cov{{\operatorname{Cov}\,}}
\def\diam{{\operatorname{diam}\,}}
\def\Leb{{\operatorname{Leb}\,}}
\def\Lip{{\operatorname{Lip}\,}}
\def\Var{{\operatorname{Var}\,}}

\def\bE{{\mathbb{E}}}
\def\bN{{\mathbb{N}}}
\def\N{{\mathbb{N}}}
\def\bP{{\mathbb{P}}}
\def\bR{{\mathbb{R}}}

\def\bZ{{\mathbb{Z}}}

\def\ta{{\tilde{a}}}
\def\tb{{\tilde{b}}}
\def\tg{{\tilde{g}}}

\def\td{{\tilde{d}}}
\def\ts{{\tilde{s}}}

\def\tnu{{\tilde{\nu}}}
\def\tS{{\tilde{S}}}
\def\tW{{\tilde{W}}}

\def\tpsi{{\widetilde{\psi}}}
\def\tsigma{{\tilde{\sigma}}}
\def\tOmega{{\tilde{\Omega}}}

\def\eps{{\varepsilon}}

\def\cA{{\mathcal{A}}}

\def\cX{{\mathcal X}}

\let\originalleft\left
\let\originalright\right
\renewcommand{\left}{\mathopen{}\mathclose\bgroup\originalleft}
\renewcommand{\right}{\aftergroup\egroup\originalright}

\title{Rates in almost sure invariance principle
for slowly mixing dynamical systems}
\author{C.~Cuny{\footnote{Universit\'e de Brest, LMBA, UMR CNRS 6205. Email: christophe.cuny@univ-brest.fr}}, J.~Dedecker{\footnote{Universit\'e Paris Descartes, Sorbonne Paris Cit\'e,  Laboratoire MAP5 (UMR 8145). Email: jerome.dedecker@parisdescartes.fr}}, A.~Korepanov{\footnote{Mathematics Institute, University of Warwick, Coventry, CV4 7AL, UK. Email: a.korepanov@warwick.ac.uk}}, Florence Merlev\`ede{\footnote{Universit\'{e} Paris-Est, LAMA (UMR 8050), UPEM, CNRS, UPEC. Email: florence.merlevede@u-pem.fr}}}
\date{12 November 2018}

\begin{document}

\maketitle

\begin{abstract}
We prove the one-dimensional almost sure invariance principle with essentially optimal rates for slowly (polynomially) mixing deterministic dynamical systems, such as Pomeau-Manneville intermittent maps,
with H\"older continuous observables.

Our rates have form $o(n^\gamma L(n))$, where $L(n)$ is a slowly varying function and
$\gamma$ is determined by the speed of mixing.
We strongly improve previous results where the best available
rates did not exceed $O(n^{1/4})$.

To break the $O(n^{1/4})$ barrier, we represent the dynamics as a Young-tower-like 
Markov chain and adapt the methods of Berkes-Liu-Wu and
Cuny-Dedecker-Merlev\`ede on the Koml\'os-Major-Tusn\'ady approximation
for dependent processes.
\end{abstract}

\noindent{\it Keywords:}  Strong invariance principle, 
KMT approximation, Nonuniformly expanding dynamical systems, Markov chain.

\smallskip

\noindent{\it MSC: } 60F17, 37E05.

\section{Introduction and statement of results}
\label{LSVmap}

In their study of turbulent bursts, Pomeau and Manneville \cite{PM80}
introduced simple dynamical systems, exhibiting intermittent transitions
between ``laminar'' and ``turbulent'' behaviour. 
Over the last few decades, such maps have been very popular
in dynamical systems.
We consider a version of Liverani, Saussol and Vaienti \cite{LSV99},
where for a fixed \(\gamma \in (0,1)\), the map
\(f \colon [0,1] \to [0,1]\) is given by
\beq \label{eq:LSV}
 f(x) = \begin{cases}
    x(1 + 2^\gamma x^\gamma), & x \leq 1/2 \\
    2 x - 1, & x > 1/2
  \end{cases}
\eeq

There exists a unique absolutely continuous $f$-invariant probability measure
$\mu$ on $[0,1]$, which is equivalent to the Lebesgue measure.

The intermittent behaviour comes from the fact that $0$ is a fixed point with 
$f'(0) = 1$.
Hence if a point $x$ is \emph{close} to $0$, then its orbit 
$(f^n(x))_{n \geq 0}$ stays around $0$ for a \emph{long} time.
The degree of intermittency is given by the parameter $\gamma$ and
is quantified by choosing an interval away from $0$ such as
$Y = ]1/2,1]$ and considering the first return time $\tau \colon Y \to \bN$,
\[ 
  \tau(x) = \min \{n \geq 1 \colon f^n(x) \in Y\}
  \, .
\]
It is straightforward to verify \cite{G04t,Y99} that for some $C > 0$
all $n \geq 1$,
\beq
  \label{eq:LSV-tails}
  C^{-1} n^{- 1/\gamma} 
  \leq \Leb(\tau \geq n) 
  \leq C n^{- 1/\gamma}
  \, ,
\eeq
where $\Leb$ denotes the Lebesgue measure on $Y$.



Suppose that $\varphi \colon [0,1] \to \bR$ is a H\"older continuous observable
with $\int \varphi \, d\mu = 0$
and let
\[
  S_n(\varphi) = \sum_{k=0}^{n-1} \varphi \circ f^k
  .
\]
We consider $S_n(\varphi)$ as a discrete time random process
on the probability space $([0,1], \mu)$.
Since $\mu$ is $f$-invariant, the increments $(\varphi \circ f^n)_{n \geq 0}$ are stationary.
Using the bound~\eqref{eq:LSV-tails}, Young~\cite{Y99} proved that the correlations decay polynomially:
\beq \label{covinesebfirst}
  \Bigl |  \int \varphi \; \varphi \circ f^n \, d\mu \Bigr |
  = O\bigl( n^{-( 1-\gamma ) / \gamma} \bigr) \, .
\eeq

If $\gamma < 1/2$, then 
$S_n(\varphi)$ satisfies the central limit theorem (CLT),
that is $n^{-1/2} S_n(\varphi)$ converges in distribution to a normal random variable with
variance
\beq \label{defc2}
  c^2 
  = \int \varphi^2 \, d\mu + 2 \sum_{n=1}^\infty \int \varphi \, \varphi \circ f^n\, d\mu
\,   .
\eeq
By~\eqref{covinesebfirst}, the series above converges absolutely. 
The asymptotics in~\eqref{covinesebfirst} is sharp \cite{G04,G04t,H04,S02,Y99}, and
for each $\gamma \geq 1/2$ there are observables $\varphi$ for which the series for
$c^2$ diverges, and the CLT does not hold.
We are interested in the case when the CLT holds, so
from here on we restrict to $\gamma < 1/2$.

In parallel with \eqref{eq:LSV}, we consider a very similar map
\beq \label{eq:HG}
 f(x) = \begin{cases}
    x(1 + x^\gamma \rho(x)), & x \leq 1/2 \\
    2 x - 1, & x > 1/2
  \end{cases}
  ,
\eeq
where, following Holland~\cite{H05} and Gou\"ezel~\cite{G04t},
$\rho \in \mathcal{C}^2((0,1/2], (0, \infty))$ is
slowly varying at $0$ and satisfies:
\begin{itemize}
  \item $x \rho'(x) = o(\rho(x))$ and $x^2 \rho''(x) = o(\rho(x))$;
  \item $f(1/2) = 1$ and $f'(x) >1$  for all $x \neq 0$;
  \item
    \(
      \displaystyle \int_0^{1/2} \frac{1}{x (  \rho(x) )^{1/\gamma}} \, dx < \infty \, .
    \)
\end{itemize}
For example, $\rho(x) = C |\log x|^{(1+ \eps) \gamma}$ with
$\eps > 0$ and $C = 2^\gamma (\log 2)^{-( 1+\eps) \gamma}$. 

Then in place of the bound $\Leb(\tau \geq n) \leq C n^{-1/\gamma}$ in 
\eqref{eq:LSV-tails} we have a slightly stronger bound
\cite[Thm~1.4.10, Prop.~1.4.12, Lem.~1.4.14]{G04t}:
\beq
  \label{eq:HG-tails}
  \int_Y \tau^{1/\gamma} \, d\Leb
  < \infty
  \, .
\eeq

\begin{rmk}
  The analysis above for the map \eqref{eq:LSV} applies to the map~\eqref{eq:HG}
  with minor differences: the correlations decay slightly faster and
  the CLT holds also for $\gamma = 1/2$ (see \cite{G04t}).
\end{rmk}

Further we use $f$ to denote either of the maps~\eqref{eq:LSV} 
and~\eqref{eq:HG}, specifying which one we refer to where it
makes a difference.

A strong generalization of the CLT and the aim of our work is the following property:

\begin{defn}
  We say that a real-valued random process $(S_n)_{n \geq 1}$ satisfies
  the \emph{almost sure invariance principle}  (ASIP)
  (also known as a \emph{strong invariance principle})
  with rate $o(n^\beta)$, $\beta \in (0,1/2)$, and variance $c^2$
  if one can redefine $(S_n)_{n \geq 1}$ without changing its distribution on a (richer) probability space on which
  there exists a Brownian motion $(W_t)_{t \geq 0}$ with variance $c^2$ such that
  \[
    S_n = W_n + o(n^\beta)
    \quad \text{almost surely.}
  \]
  We define similarly the ASIP with rates $o(r_n)$ or $O(r_n)$
  for deterministic sequences $(r_n)_{n \geq 1}$.
\end{defn}

For the map~\eqref{eq:LSV} with H\"older continuous observables $\varphi$, the ASIP for
$S_n(\varphi)$ has been first proved by Melbourne and Nicol \cite{MN05}, albeit 
without explicit rates. In \cite[Thm.~1.6 and Rmk.~1.7]{MN09}, 
the same authors obtained the ASIP with rates
\[
  S_n(\varphi) - W_n =
  \begin{cases}
    o(n^{\gamma/2 + 1/4 + \eps}), & \gamma \in ]1/4, 1/2[ \\
    o(n^{3/8 + \eps}), & \gamma \in ]0, 1/4]
  \end{cases}
\]
for all $\eps > 0$.
Their proof is based on Philipp and Stout \cite[Thm.~7.1]{PS}.
This result has been subsequently improved. Using the approach for the
reverse martingales of Cuny and Merlev\`ede \cite{CM15},
Korepanov, Kosloff and Melbourne \cite{KKM} proved the ASIP with rates
\[
  S_n(\varphi) - W_n =
  \begin{cases}
    o(n^{\gamma + \eps}), & \gamma \in [1/4, 1/2[ \\
    O ( n^{1/4} (\log n)^{1/2} (\log \log n)^{1/4} ), & \gamma \in ]0, 1/4[
  \end{cases}
\]
for all $\eps > 0$. (Subsection~\ref{sec:R-ASIP_NUE} provides some more details.)

When $\varphi$ is not H\"older continuous, the situation is more delicate. For instance, functions with discontinuities are not easily amenable to the method of Young towers used in \cite{KKM,MN05,MN09}.
For $\varphi$ of bounded variation, using the conditional quantile method, Merlev\`ede and Rio \cite{MR12} proved the ASIP with rates
\[
  S_n(\varphi) - W_n =
  O ( n^{\gamma'} (\log n)^{1/2} (\log \log n)^{ (1+ \eps) \gamma'} )
\]
for all $\eps > 0$, where $\gamma' = \max\{\gamma, 1/3\}$.
Besides considering observables of bounded variation, the results of \cite{MR12} also cover a large class of
unbounded observables.

In all the papers above, the rates are not better than $O(n^{1/4})$,
which could be perceived as largely suboptimal when $0<\gamma<1/4$ due to the intuition
coming from the processes with iid increments \cite{KMT75} and 
recent related work \cite{BLW14,CDM17}. Our main result is:

\begin{thm} 
  \label{ThASIPLSV}
  Let $\gamma \in (0,1/2)$ and $\varphi \colon [0,1] \to \bR$ be a H\"older continuous observable
  with $\int \varphi \, d\mu = 0$. 
  For the map~\eqref{eq:LSV},
  the random process $S_n(\varphi)$ satisfies the ASIP
  with variance $c^2$ given by \eqref{defc2} and rate
  $o(n^{\gamma} (\log n)^{ \gamma + \eps} )$ for all $\eps >0$.
  For the map~\eqref{eq:HG},
  the random process $S_n(\varphi)$ satisfies the ASIP
  with variance $c^2$ given by \eqref{defc2} and rate
  $o(n^{\gamma})$.

\end{thm}

The rates in Theorem~\ref{ThASIPLSV} are optimal in the following sense:
\begin{prop}
  \label{lem:OR}
  Let $f$ be the map \eqref{eq:LSV}. There exists a H\"older continuous
  observable $\varphi$ with $\int \varphi \, d \mu = 0$ such that
  \[
    \limsup_{n \to \infty} (n \log n)^{-\gamma} |S_n(\varphi) - W_n|
    > 0
  \]
  for all Brownian motions $(W_t)_{t \geq 0}$ defined on the same (possibly enlarged)
  probability space as $(S_n(\varphi))_{n \geq 0}$.
  Hence, one cannot take $\eps=0$ in Theorem \ref{ThASIPLSV}.
\end{prop}


\begin{rmk}\label{cob}
If $c^2=0$, the rate in the ASIP can be improved to $O(1)$.
Indeed, then it is well-known that $\varphi$ is a {\it coboundary} in the sense that
$\varphi= u- u \circ f$ with some $u \colon [0,1] \to \bR$. By \cite[Prop.~1.4.2]{G04t}, 
$u$ is bounded, thus $S_n(\varphi)$ is bounded uniformly in $n$.
\end{rmk}

\begin{rmk} \label{rmk:1-6}
It is possible to relax the assumption that $\varphi$ is H\"older continuous.
As a simple example, Theorem~\ref{ThASIPLSV} holds if $\varphi$ is H\"older on $(0,1/2)$ 
and on $(1/2,1)$, with a discontinuity at $1/2$.
See Subsection \ref{piecewise-holder} for further extensions.
\end{rmk}

\begin{rmk}
Intermittent maps are prototypical examples
of \emph{nonuniformly expanding dynamical systems,}
to which our results apply in a general setup, and so does the discussion
of rates preceding Theorem~\ref{ThASIPLSV}.
We focus on the maps~\eqref{eq:LSV} and~\eqref{eq:HG} for simplicity only,
and discuss the generalization in Section~\ref{sec:NUE}. 
\end{rmk}

The paper is organized as follows. In Section~\ref{sec:RMC},
following Korepanov~\cite{K17}, we represent the dynamical systems~\eqref{eq:LSV}
and~\eqref{eq:HG} as a function of the trajectories of a particular Markov chain; further, we introduce a
\emph{meeting time} related to the Markov chain and estimate its moments.
In  Section~\ref{sec:proofASIP} we prove Theorem~\ref{ThASIPLSV}
for our new process (which is a function of the whole future trajectories 
of the Markov chain) by adapting the ideas of 
Berkes, Liu and Wu~\cite{BLW14} and Cuny, Dedecker and Merlev\`ede~\cite{CDM17}.
In Section~\ref{sec:NUE} we generalize our results to the class of
nonuniformly expanding dynamical systems and show the optimality of the rates.

Throughout, we use the notation $a_n \ll b_n$ and $a_n = O(b_n)$ interchangeably,
meaning that there exists a positive  constant $C$ not depending on $n$ such that 
$a_n \leq  Cb_n $ for all sufficiently large $n$. As usual, $a_n = o (b_n)$ means that $\lim_{n \rightarrow \infty} a_n /b_n =0$. Recall that $v : X \rightarrow \bR$ is a H\"older observable 
(with a H\"older exponent $ \eta >0$) on a bounded metric space $(X,d)$ if $\Vert v \Vert_{\eta} = \vert v \vert_{\infty} + \vert v \vert_{\eta} < \infty$ where $\vert v \vert_{\infty}  = \sup_{x \in X} |v(x)|$ and 
$ \vert v \vert_{\eta} = \sup_{x \neq y} \frac{|v(x) - v(y)|}{ d(x,y)^{\eta}} $. All along the paper, we use the notation $\bN=\{0,1,2,\ldots\}$.

\section{Reduction to a Markov chain}
\label{sec:RMC}

\subsection{Outline}

In this section we construct a stationary Markov chain $g_0, g_1, \ldots$ 
on a countable state space $S$, the space of all possible future trajectories $\Omega$
and an observable $\psi \colon \Omega \to \bR$
such that the random process $( X_n)_{n \geq 0}$ where $X_n = \psi(g_n, g_{n+1}, \ldots)$
has the same distribution as $(\varphi \circ f^n)_{n \geq 0}$, the increments
of $(S_n(\varphi))_{n \geq 1}$.

Our Markov chain is in the spirit of the classical Young towers~\cite{Y99}. 
Just as the Young towers for the maps~\eqref{eq:LSV} and \eqref{eq:HG},
our construction enjoys recurrence properties related to the choice of $\gamma$,
and we supply $\Omega$ with a metric, with respect to which $\psi$ is Lipschitz.

We follow the ideas of \cite{K17r}, though in the setup of the 
maps~\eqref{eq:LSV} and~\eqref{eq:HG} we are able to make the proofs simpler
and hopefully easier to read.

\subsection{Basic properties of intermittent maps}
\label{sec:BPIM}

A standard way to work with maps \eqref{eq:LSV}, \eqref{eq:HG} is an inducing scheme.
As in Section~\ref{LSVmap}, set \(Y = ]1/2,1]\) and let \(\tau \colon Y \to \bN\) be the \emph{inducing time,}
\(\tau(x) = \min\{ k \geq 1 \colon f^k(x) \in Y\}\).
Let \(F \colon Y \to Y\) be the \emph{induced map}, \(F(x) = f^{\tau(x)}(x)\).
Let \(\alpha\) be the partition of \(Y\) into the intervals where \(\tau\) is
constant. Let $\beta = 1/\gamma$.

We remark that $\gcd\{ \tau(a) \colon a \in \alpha \} = 1$.

Let \(m\) denote the Lebesgue measure on \(Y\), normalized so that it is
a probability measure. Recall that we have the bounds
\begin{itemize}
  \item \(m(\tau \geq n) \leq C n^{-\beta}\) for all $n \geq 1$
    for the map~\eqref{eq:LSV};
  \item $\int \tau^{\beta} \, dm < \infty$
    for the map~\eqref{eq:HG}.
\end{itemize}

The induced map \(F\) satisfies the following properties:
\begin{itemize}
  \item (full image) \(F \colon a \to Y\) is a bijection for each \(a \in \alpha\);
  \item (expansion)  there is \(\lambda > 1\) such that \(|F'| \geq \lambda\);
  \item (bounded distortion) there is a constant \(C_d \geq 0\) such that
    \[
      \bigl| \log |F'(x)| - \log |F'(y)| \bigr|
      \leq C_d |F(x) - F(y)|
    \]
    for all \(x,y \in a\), \(a \in \alpha\).
\end{itemize}

\subsection{Disintegration of the Lebesgue measure}

The properties in Subsection~\ref{sec:BPIM} allow a disintegration of the
measure \(m\), as described in this subsection.

Let \(\cA\) denote the set of all finite words in the alphabet \(\alpha\),
not including the empty word. For \(w = a_0 \cdots a_{n-1} \in \cA\),
let \(|w| = n\) and let \(Y_w\) denote the cylinder
of points in \(Y\) which follow the itinerary of letters of \(w\) under
the iteration of \(F\):
\[
  Y_w = \{y \in Y \colon F^{k}(y) \in a_k \text{ for } 0 \leq k \leq n-1 \}
  .
\]
Let also \(h \colon \cA \to \bN\), \(h(w) = \tau(a_0) + \cdots + \tau(a_{n-1})\)
for \(w = a_0 \cdots a_{n-1}\).

For \(w_0, \ldots, w_n \in \cA\), let
\(w_0 \cdots w_n \in \cA\) denote the concatenation.

\begin{prop}
  \label{prop:anqw}
  For each infinite sequence \(a_0, a_1, \ldots \in \alpha\), there exists a unique
  \(y \in Y\) such that \(F^n(y) \in a_n\) for all \(n \geq 0\).
  
  In particular, for each sequence
  \(w_0, w_1, \ldots \in \cA\) there exists a unique
  \(y \in Y\) such that
  \(y \in Y_{w_0}\), \(F^{|w_0|}(y) \in Y_{w_1}\),
  \(F^{|w_0| + |w_1|}(y) \in Y_{w_2}\),
  and so on.
\end{prop}

\begin{proof}
  Uniqueness of $y$ follows from expansion of $F$, so it is enough to show
  existence.
  
  Let $w_n = a_0 \cdots a_{n-1}$.
  Note that $Y_{w_n}$, $n \geq 0$, is a nested sequence of intervals
  with shrinking to $0$ length, closed on the right and
  open on the left. Let $y$ be the only point in the intersection
  of their closures,
  \(
    \{y\} = \cap_n \bar{Y}_{w_n}
  \).
  
  Suppose that $y \not \in \cap_n Y_{w_n}$. Then $y \in \bar{Y}_{w_n} \setminus Y_{w_n}$
  for some $n$, thus $y$ is a left end-point of $Y_{w_n}$.
  Observe that $\bar{Y}_{w_{n+1}}$ is contained in $\bar{Y}_{w_n}$
  but cannot contain its left end-point, i.e.\ $y \not \in \bar{Y}_{w_{n+1}}$.
  This is a contradiction, proving that $y \in \cap_n Y_{w_n}$.
  Hence $F^n(y) \in a_n$ for all $n$, as required.
\end{proof}

\begin{prop}
  \label{prop:nai}
  There exist a probability measure \(\bP_\cA\) on \(\cA\) and a disintegration
  \[
    m = \sum_{w \in \cA} \bP_\cA(w) m_w,
  \]
  where
  \begin{itemize}
    \item each \(m_w\) is a probability measure supported on \(Y_w\);
    \item \((F^{|w|})_* m_w = m\);
    \item $\bP_\cA(w) > 0$ for each $w$;
    \item for the map \eqref{eq:LSV}, \(\bP_\cA(h \geq k) \leq C_\beta k^{-\beta}\) for all \(k \geq 1\),
      where \(C_\beta > 0\) is a constant;
    \item for the map \eqref{eq:HG}, \(\int h^\beta \, d\bP_\cA < \infty\).
  \end{itemize}
\end{prop}

The disintegration in Proposition~\ref{prop:nai} was introduced in
\cite{Z09} and called \emph{regenerative partition of unity}.
The bounds on the tail of $h$ are proved in \cite{K17}.
This disintegration is the basis of the Markov chain construction.

\subsection{Construction of the  Markov chain} \label{constructionMC}

Let \(g_0, g_1, \ldots\) be a Markov chain with state space
\[
  S = \{(w, \ell) \in \cA \times \bZ \colon 0 \leq \ell < h(w)\}
\]
and transition probabilities
\begin{equation}
  \label{eq:trpr}
  \begin{aligned}
  \bP (g_{n+1} = (w, \ell) & \mid g_n = (w', \ell'))
  \\ & = \begin{cases}
    1, & \ell = \ell' + 1 \text{ and } \ell' + 1 < h(w) \text{ and } w=w' \\
    \bP_\cA(w) , & \ell = 0 \text{ and } \ell' + 1 = h(w') \\
    0, & \text{else}
  \end{cases}
  \end{aligned}
\end{equation}
The Markov chain \(g_0,g_1, \ldots\) has a unique (hence ergodic) invariant probability measure $\nu$ on \(S\), given by 
\beq \label{defnu}
\nu ( w, \ell ) = \frac{ \bP_{\cA} (w) {\bf 1}_{\{0 \leq \ell < h(w) \}}}{  \sum_{ (w , \ell) \in S } \bP_{\cA} (w)  } =  \frac{\bP_{\cA} (w) {\bf 1}_{\{0 \leq \ell < h(w) \}} }{   \bE_{\cA} (h)  }  \, .
\eeq
The Markov chain  \(g_0,g_1, \ldots\) starting from $\nu$ defines a probability measure \(\bP_\Omega\) on
the space \(\Omega \subset S^{\bN}\) of sequences which correspond to
non-zero probability transitions.
Let \(\sigma \colon \Omega \to \Omega\) be the left shift action,
\[
  \sigma (g_0, g_1, \ldots) = (g_1, g_2, \ldots) \,   .
\]

\begin{rmk} \label{remarkonaperiodicity}
  There exists $w \in \cA$ with $\bP_\cA(w) > 0$ and $h(w) = 1$.
  Therefore, the Markov chain \(g_0, g_1, \ldots\) is aperiodic. Aperiodicity is used in the proof of
  the ASIP (namely, in the proof of Lemma \ref{lem:nngh} to apply Lindvall's result \cite{Li79}).
  However, in the general case, as far as the ASIP is concerned,  aperiodicity is not necessary (see Section~\ref{sec:NUE}).
\end{rmk}

We supply the space \(\Omega\) with a \emph{separation time}
\(s \colon \Omega \times \Omega \to \bN \cup \{\infty\}\),
measured in terms of the number of visits to
\(S_0 = \{(w,\ell) \in S \colon \ell = 0\}\) as follows.
For \(a,b \in \Omega\),
\begin{equation} \label{eq:septwop}
\begin{aligned}
  a & = (g_0, \ldots, g_N, g_{N+1}, \ldots), \\
  b & = (g_0, \ldots, g_N, g'_{N+1}, \ldots) 
\end{aligned}
\end{equation}
with \(g_{N+1} \neq g'_{N+1}\), we set
\[
  s(a,b)
  = \# \{ 0 \leq  n \leq N \colon g_n \in S_0 \}
  .
\]
We define a \emph{separation metric} \(d\) on \(\Omega\) by
\begin{equation} \label{eq:sepd}
  d(a,b) = \lambda^{-s(a,b)}
  .
\end{equation}

For \(g = (w, \ell) \in S\), define \(X_g \subset [0,1]\), 
\(X_g = f^\ell (Y_w)\).
Then, similar to Proposition~\ref{prop:anqw},
to each \((g_0, g_1, \ldots) \in \Omega\) there corresponds a unique
\(x \in [0,1]\) such that \(f^n(x) \in X_{g_n}\) for all \(n \geq 0\)
(but for a given \(x\), there may be many such
\((g_0, g_1, \ldots) \in \Omega\)).

Thus we introduce a projection \(\pi \colon \Omega \to [0,1]\),
with \(\pi(g_0, g_1, \ldots) = x\) where
\(f^n(x) \in X_{g_n}\) for all \(n \geq 0\) as above.

The key properties of the projection \(\pi\) are:
\begin{lemma}
  \label{lem:proj}
  \
  \begin{itemize}
     \item \(\pi\) is Lipschitz:
      \(|\pi(a) - \pi(b)| \leq  d(a,b)\) for all \(a,b \in \Omega\);
    \item \(\pi\) is a measure preserving map between the 
      probability spaces \((\Omega, \bP_\Omega)\)
      and \(([0,1], \mu)\);
    \item \(\pi\) is a semiconjugacy between \(\sigma \colon \Omega \to \Omega\)
      and \(f \colon [0,1] \to [0,1]\), i.e.\ the following diagram
      commutes:
      \[
        \begin{tikzcd}[row sep=2.5em]
          \Omega \arrow[r,"\sigma"] \arrow[d,swap,"\pi"] 
          &
          \Omega \arrow[d,swap,"\pi"] \\
          {[0,1]} \arrow[r,"f"] & {[0,1]}
        \end{tikzcd}
      \]
  \end{itemize}
\end{lemma}

\begin{cor} 
  \label{cor:proj}
  Suppose that $\varphi \colon [0,1] \to \bR$ is H\"older continuous.
  Let $\psi = \varphi \circ \pi$
  and $X_k = \psi(g_k, g_{k+1}, \ldots)$ for $k \geq 0$.
  Then
  \begin{enumerate}[label=(\alph*)]
    \item $\psi$ is H\"older continuous.
    \item The process $(X_k)_{k \geq 0}$ on the probability space
      $(\Omega, \bP_\Omega)$ is equal in law to
      $(\varphi \circ f^k)_{k \geq 0}$ on $([0,1], \mu)$.
  \end{enumerate} 
\end{cor}

\subsection{Proof of Lemma~\ref{lem:proj}}

The last item, namely, the property that \(\pi \circ \sigma = f \circ \pi\) follows directly
from the construction of \(\sigma\) and \(\pi\).

\smallskip

We prove now the first item.
Suppose that \(a,b \in \Omega\) are as in \eqref{eq:septwop} and write
\begin{align*}
  g_0, & \ldots, g_N = \\
  & (w_0, \ell_0), \ldots, (w_0, h(w_0) - 1), \ 
  (w_1, 0), \ldots, (w_1, h(w_1) - 1),  \ldots,
  (w_k, 0), \ldots, (w_k, \ell_k) \, , 
\end{align*}
where $0 \leq \ell_0 <  h(w_0)$, $0 \leq \ell_k < h(w_k)$ and 
$h(w_0) -\ell_0 +\sum_{i=1}^{k-1} h(w_i) + \ell_k =N$. 
Then both \(\pi(a)\) and \(\pi(b)\) belong to
\(f^{\ell_0}(Y_{w_0 \cdots w_k})\). 

Suppose that \(\ell_0 \neq 0\). Then \(s(a,b) = k\).
Since \(|f'| \geq 1\) and \(|F'| \geq \lambda\), 
\[
  \diam f^{\ell_0}(Y_{w_0 \cdots w_k}) 
  \leq \diam Y_{w_1 \cdots w_k}
  \leq \lambda^{-k}
  .
\]
Then \(|\pi(a) - \pi(b)| \leq \lambda^{-s(a,b)} = d(a,b)\).
If \(\ell_0 = 0\), then \(s(a,b)=k+1\) and
\(\diam f^{\ell_0}(Y_{w_0 \cdots w_k}) \leq \lambda^{-(k+1)}\).
Again, \(|\pi(a) - \pi(b)| \leq d(a,b)\), as required.

\medskip
  
 It remains to prove the second item, namely:   \(\pi_* \bP_\Omega = \mu\).  Let \(\Omega_0 = \{(g_0, g_1, \ldots) \in \Omega \colon g_0 \in S_0\}\).
Then \(\bP_\Omega (\Omega_0) > 0\). Let 
\[
  \bP_{\Omega_0} (\cdot)
  = \frac{\bP_\Omega( \cdot \cap \Omega_0)}{\bP_\Omega (\Omega_0)}
\]
be the corresponding conditional probability measure. We shall use the following intermediate result whose proof is given later. 

\begin{prop}
  \label{prop:pim}
  \(\pi_* \bP_{\Omega_0} = m\).
\end{prop}
Let us complete the proof of the second item with the help of this proposition. 
  Note that \(\sigma \colon \Omega \to \Omega\)
  preserves the ergodic probability measure
  \(\bP_\Omega\). Since \(f \circ \pi = \pi \circ \sigma\),
  the measure \(\upsilon := \pi_* \bP_\Omega\) on \([0,1]\) is
  \(f\)-invariant and ergodic, as is \(\mu\).
  
  Suppose that \(\upsilon\) and \(\mu\) are different measures.
  Since they are both \(f\)-invariant and ergodic,
  they are singular with respect to each other:  
  there exists \(A \subset [0,1]\) such that
  \(\mu(A) = 1\) and \(\upsilon(A) = 0\).

  Let \(\upsilon|_Y\) and \(\mu|_Y\) denote the restrictions on \(Y\).
  By Proposition~\ref{prop:pim}, \(m \ll \upsilon|_Y\).
  Since in turn \(\mu|_Y \ll m\), it follows that
  \(\mu|_Y \ll \upsilon|_Y\).
  Hence 
  \(
    \mu(A \cap Y) 
    = \upsilon(A \cap Y)
    = 0
  \).
  Also, \(\mu(Y \setminus A) = 0\), so \(\mu(Y) = 0\),
  which contradicts  the fact that \(\mu\) is
  equivalent to the Lebesgue measure on \([0,1]\).
  Thus \(\mu = \upsilon\).  
  
  \smallskip
  
To end the proof of the second item, it remains to show Proposition~\ref{prop:pim}. 

\medskip

\noindent{\it Proof of Proposition~\ref{prop:pim}.}
  Our strategy is to show that for each \(w \in \cA\),
  \[
    \bP_{\Omega_0} (\pi^{-1}(Y_w))
    = m(Y_w)
   \,  .
  \]
  Then the result follows from  Carath\'eodory's extension theorem.
  
  Let \(m = \sum_{w \in \cA} \bP_\cA(w) m_w\) be the decomposition
  from Proposition~\ref{prop:nai}. Recall that each
  \(m_w\) is supported on \(Y_w\) and \((F^{|w|})_* m_w = m\).
  Since \(F^{|w|} \colon Y_w \to Y\) is a diffeomorphism between
  two intervals, the measures \(m_w\) are uniquely determined by these
  properties. It is straightforward to write
  \(m_w = \sum_{w' \in \cA} \bP_\cA(w') m_{w w'}\)
  for each \(w\). (Here \(w w'\) is the concatenation of \(w, w'\) and the
  measures \(m_{w w'}\) are from the same decomposition.)
  Thus we obtain a decomposition
  \[
    m = \sum_{w, w' \in \cA} \bP_\cA(w) \bP_\cA(w') m_{w w'}
  \,   .
  \]
  Further, for \(n \geq 0\), we write
  \[
    m = \sum_{w_0, \ldots, w_n \in \cA} \bP_\cA(w_0) \cdots \bP_\cA(w_n)
    m_{w_0 \cdots w_n}
  \,   .
  \]
  
  Suppose that \(w \in \cA\) with \(|w|=n+1\). For every \(w_0, \ldots, w_n \in \cA\),
  either \(Y_{w_0 \cdots w_n} \subset Y_w\) (when the word \(w_0 \cdots w_n\)
  starts with \(w\)) or
  \(Y_{w_0 \cdots w_n} \cap Y_w = \emptyset\) (otherwise).
  Hence
  \begin{equation}
    \label{eq:rrt1}
    m(Y_w)
    = \sum_{\substack{
        w_0, \ldots, w_n \in \cA \colon \\ 
        Y_{w_0 \cdots w_n} \subset Y_w
      }}
      \bP_\cA(w_0) \cdots \bP_\cA(w_n)
  \,   .
  \end{equation}
  
  For \(w_0, \ldots, w_n\), let
  \(\Omega_{w_0, \ldots, w_n}\)
  denote the subset of \(\Omega_0\) with the first coordinates
  \[
    (w_0,0), \ldots, (w_0, h(w_0) - 1),
    \ldots,
    (w_n,0), \ldots, (w_n, h(w_n) - 1)
 \,    .
  \]
  Note that \(\pi(\Omega_{w_0, \ldots, w_n}) = Y_{w_0 \cdots w_n}\)
  and by~\eqref{eq:trpr},
  \[
    \bP_{\Omega_0} (\Omega_{w_0, \ldots, w_n})
     = \bP_\cA(w_0) \cdots \bP_\cA(w_n)
  \,    .
  \]
  Then
  \begin{equation}
    \label{eq:rrt2}
    \bP_{\Omega_0}(\pi^{-1} (Y_w))
    = \sum_{\substack{
        w_0, \ldots, w_n \in \cA \colon \\ 
        Y_{w_0 \cdots w_n} \subset Y_w
      }}
      \bP_{\Omega_0} (\Omega_{w_0, \ldots, w_n})
    = \sum_{\substack{
        w_0, \ldots, w_n \in \cA \colon \\ 
        Y_{w_0 \cdots w_n} \subset Y_w
      }}
      \bP_\cA(w_0) \cdots \bP_\cA(w_n)
  \,   .
  \end{equation}
  
  Combining~\eqref{eq:rrt1} and~\eqref{eq:rrt2}, we obtain that
  \(\bP_{\Omega_0}(\pi^{-1} (Y_w)) = m(Y_w)\), as required. \qed

\section{Meeting time}
\label{sec:MT}

In Section~\ref{sec:RMC} we constructed the stationary and aperiodic Markov chain $(g_n)_{n \geq 0}$.
In this section we introduce a meeting time on it and
use it to prove a number of statements which shall play
a central role in the proof of the ASIP.


We work with the notation of Section~\ref{sec:RMC}.
Without changing the distribution, we redefine the Markov chain \(g_0,g_1, \ldots\) on a new
probability space as follows. Let \(g_0 \in S\) be distributed according to \(\nu\) 
(the stationary distribution defined by \eqref{defnu}).
Let \(\eps_1, \eps_2, \ldots\) be a sequence of independent
identically distributed random variables with values in \(\cA\), distribution
\(\bP_\cA\), independent from \(g_0\). For \(n \geq 0\) let
\beq \label{defgnviaU}
  g_{n+1} 
  = U(g_n, \eps_{n+1}) \, , 
\eeq
where
\beq \label{defUforMC} 
  U((w, \ell), \eps)
  = \begin{cases}
    (w, \ell+1), & \ell < h(w) - 1 \, ,  \\
    (\eps, 0), & \ell = h(w) - 1 \, .
  \end{cases}
\eeq
We refer to \((\eps_n)_{n \geq 1}\) as \emph{innovations}. 

Let \(g_0^*\) be a random variable in \(S\) with distribution $\nu$,
independent from \(g_0\) and \((\eps_n)_{n \geq 1}\).
Let \(g_0^*, g_1^*, g_2^*, \ldots\) be a Markov chain given by
\beq \label{defMCstar}
  g_{n+1}^* = U(g_n^*, \eps_{n+1})
  \ \text{ for } \ n \geq 0 \, .
\eeq
Thus the chains \( (g_n)_{n \geq 0}\) and \(( g_n^*)_{n \geq 0}\)
have independent initial states, but share the same innovations.
Define the meeting time:
\beq \label{defofT}
  T = \inf \{n \geq 0 \colon g_n = g_n^* \} \, .
\eeq

For $\beta, \eta > 1$, define $\psi_{\beta, \eta}, \tpsi_{\beta, \eta} \colon [0,\infty) \to [0,\infty)$,
\[
  \psi_{\beta, \eta} (x) = x^{\beta} ( \log ( 1+x) )^{- \eta}
  \, ,
  \qquad
  \tpsi_{\beta, \eta}(x) = x^{\beta - 1} ( \log ( 1+x) )^{- \eta}
\]
for $x > 0$ and $\psi_{\beta, \eta}(0) = \tpsi_{\beta, \eta}(0) = 0$.

For the maps~\eqref{eq:LSV} and~\eqref{eq:HG}, moments of $T$ can be estimated by
Proposition~\ref{prop:nai} and the following lemma:

\begin{lemma} \label{lem:nngh}
  Suppose that $\beta > 1$.
  \begin{enumerate}[label=(\alph*)]
    \item\label{lem:nngh:w} If \(\bP_\cA(h \geq k) \ll k^{-\beta}\), then
      \(\bE (\tpsi_{\beta, \eta} (T) ) < \infty\)
      for all $\eta >1$.
    \item\label{lem:nngh:s} If \(\int h^\beta \, d\bP_\cA < \infty\), then
      $\bE ( T^{\beta -1} ) < \infty$.
  \end{enumerate}
\end{lemma}

\begin{proof}
  Let \(S_c = \{(w,\ell) \in S \colon \ell = h(w) - 1\}\) be the ``ceiling'' of $S$ and
  \[
    T^* = \inf \{n \geq 0 \colon g_n \in S_c \text{ and } g_n^* \in S_c  \}
    \, .
  \]
  From the representation \eqref{defgnviaU}, it is clear that
  $T \leq T^* + 1$.
  
  Now, the segments $(g_0, g_1, \ldots, g_{T^*} )$ and $(g^*_0, g^*_1, \ldots, g^*_{T^*} )$ 
  never use the same innovations and behave independently. In addition,
  $g_{T^*+1}= g^*_{T^*+1} = (\eps_{T^* +1},0)$ and 
  $g_{n + T^* } = g^*_{n + T^* } $ for any $n \geq 1$. 

  Consider $(\eps'_n)_{n \geq 1}$, an independent copy of $(\eps_n)_{n \geq 1}$, 
  independent also from $g_0$. Let \(g_0'\) be a random variable in \(S\) with distribution $\nu$, 
  independent from $(g_0, (\eps_n)_{n \geq 1}, (\eps'_n)_{n \geq 1})$.
  Define the Markov chain $(g'_n )_{n \geq 0}$ by
  \[
    g_{n+1}'= U(g_n', \eps'_{n+1})
    \ \text{ for } \ n \geq 0
   \, .
  \]
  Let
  \[
    T' = \inf \{n \geq 0 \colon g_n \in S_c \text{ and } g_n' \in S_c  \} \, .
  \]
  Due to the previous considerations, $T'$ is equal to $T^*$ in law.
  
  Note that $S_c$ is a recurrent atom for the Markov chain $(g_n)_{n \geq 0}$. Let
  \[
    \tau_0 = \inf \{ n \geq 0 \colon g_n \in S_c \}
  \]
  be the first renewal time. If \(\bP_\cA(h \geq k) \ll k^{-\beta}\), we claim that for all $\eta >1$,
  \[
    \bE  (\tpsi_{\beta, \eta} (\tau_0) ) < \infty 
    \, .
  \]
  Then, according to Lindvall \cite{Li79} (see also Rio \cite[Prop.~9.6]{Ri00}),
  since the chain $(g_n)_{n \geq 0}$ is aperiodic (see Remark \ref {remarkonaperiodicity}),
  $\bE  (\tpsi_{\beta, \eta} (T') ) < \infty$ and~\ref{lem:nngh:w} follows.
  For~\ref{lem:nngh:s}, the argument is similar, with 
  $x^{\beta}$ instead of $\psi_{\beta, \eta} (x)$ and
  $x^{\beta-1}$ instead of $\tpsi_{\beta, \eta} (x)$.

  It remains to verify the claim.
  Note that if $g_0 = (w, \ell)$, then $\tau_0 = h(w) - \ell - 1$ and
  \[
    \tpsi_{\beta, \eta} (\tau_0)
    = \frac{(h(w) - \ell - 1)^{\beta - 1}}{(\log (h(w) - \ell))^{  \eta}}
    \leq C_{\beta, \eta} \frac{h(w)^{\beta - 1}}{(\log h(w))^{  \eta}}
    .
  \]
  For any $\eta >1$, using that $\nu(w, \ell) \leq \bP_\cA(w)/ \bE_\cA (h) $, write
  \begin{align*}
    \bE  (\tpsi_{\beta, \eta} (\tau_0) )
    & = \sum_{\substack{w \in \cA, \\ 0 \leq \ell < h(w) }} 
      \bE_{g_0 = (w, \ell)}  (\tpsi_{\beta, \eta} (\tau_0) ) \nu (w, \ell) 
    \\ & \leq C_{\beta, \eta}  (\bE_\cA (h))^{-1}\sum_{w \in \cA} 
      \frac{h(w)^{\beta}}{(\log h(w))^{ \eta}} \bP_\cA(w)
    = C_{\beta, \eta} (\bE_\cA (h))^{-1} \bE_\cA (\psi_{\beta, \eta} (h))
    < \infty  \, , 
  \end{align*}
by taking into account Proposition  \ref{prop:nai}. 
\end{proof}

Let $\psi \colon \Omega \to \bR$ be a H\"older continuous observable
with $\int \psi\, d\bP_\Omega = 0$. (Such as $\psi = \varphi \circ \pi$
in Section~\ref{sec:RMC}.)
For $\ell \geq 0$, define $\delta_\ell \colon \Omega \to \bR$,
\[
  \delta_\ell (g_0, g_1, \ldots)
  = \sup \bigl| \psi(g_0, g_1, \ldots, g_{\ell+1}, g_{\ell+2}, \ldots) - 
  \psi(g_0, g_1, \ldots, \tg_{\ell+1}, \tg_{\ell+2}, \ldots) \bigr|
  \, ,
\]
where the supremum is taken over all possible trajectories
$(\tg_{\ell + 1}, \tg_{\ell+2}, \ldots)$.

\begin{prop}
  \label{prop:boundofthedelta1first}
  Assume that $\bE(T) < \infty$. For all $r \geq 1$,
  \[
    \bE (\delta_\ell)
    \ll \ell^{-r/2} + \bP (T \geq [\ell/r])
    \, .
  \]
\end{prop}

\begin{proof}
  By \eqref{eq:sepd} and the first item of Lemma \ref{lem:proj},
  there exist $C>0$ (depending on the H\"older norm of $\psi$)
  and $\theta \in (0,1)$ (depending on $\lambda$ and on the H\"older exponent of $\psi$)
  such that
  \(
    \delta_\ell \leq C \theta^{s_\ell}
  \),
  where $s_\ell = \#\{k \leq \ell \colon g_k  \in S_0\}$.
  Write
  \begin{equation}
    \label{dec1step1}
    \begin{aligned}
      C^{-1}\bE(\delta_\ell)
      & \leq
      \bE ( \theta^{s_\ell})
      \leq \theta^{ \frac{1}{2} (\ell+1) \bP (g_0 \in S_0)}
      + \bE \bigl(
        \theta^{s_\ell} \mathbf{1}_{s_\ell < \frac{1}{2} (\ell+1)\bP (g_0 \in S_0)}
      \bigr) 
      \\ & \leq \theta^{\frac{1}{2} (\ell+1) \bP (g_0 \in S_0)} 
      + \bP \Bigl(s_\ell  < \frac{1}{2} (\ell+1)\bP (g_0 \in S_0) \Bigr)
      \, .
    \end{aligned}
  \end{equation}
  Next,
  \[
    \bP \Bigl(s_\ell  < \frac{1}{2} (\ell+1)\bP (g_0 \in S_0) \Bigr) 
    \leq \bP \Bigl( \Bigl| 
      \sum_{i=0}^{\ell} {\bf 1}_{\{g_i \in S_0 \}}   - (\ell+1) \nu (S_0)
    \Bigr| > \frac{1}{2} (\ell+1) \nu (S_0) \Bigr)
    \, .
  \]
  Recall now the definition \eqref{defofT} of the meeting time $T$ and
  the following coupling inequality: for all $n \geq 1$, 
  \beq \label{couplinginebeta}
    \beta (n)
    := \frac{1}{2} \int \| \delta_{(x,y)} (P \times P)^n  - \nu \times \nu \|_v 
    \, d (\nu \times \nu) (x,y) 
    \leq  \bP  ( T \geq  n ) \, , 
  \eeq
  where $\| \cdot \|_{v}$ denotes the total variation norm of a signed measure
  and $P$ is the transition function of the Markov chain $(g_k)_{k \geq 0}$. 
  From $\bE(T) < \infty$, it follows that
  $\sum_{n \geq 1} \beta (n) < \infty$. Applying \cite[Thm.~6.2]{Ri00} and using that 
  $\alpha (n) \leq \beta (n)$, where $(\alpha (n))_{n \geq 1}$ is the sequence of strong mixing coefficients
  defined in \cite[(2.1)]{Ri00},  we infer  that for all $r \geq 1$, 
  \beq \label{rioine}
    \bP \Bigl( \Big | \sum_{i=0}^{\ell} 
      \mathbf{1}_{\{g_i \in S_0 \}} - (\ell+1) \nu (S_0)
    \Bigr| > \frac{1}{2} (\ell+1) \nu (S_0) \Bigr)
    \leq c_1 \ell^{-r/2} + c_2  \bP  ( T \geq  [\ell/r] ) \, , 
  \eeq
  where $c_1$ and $c_2$ are positive constant independent of $\ell$.
  The result follows.
\end{proof}

For $n \geq 0$, let
\[
  X_n = \psi \circ \sigma^n
  = \psi (g_n, g_{n+1}, \ldots) 
  \, .
\]
Then $(X_n)_{n \geq 0}$ is a stationary random process.
It is straightforward to use the meeting time to estimate correlations:

\begin{lemma}
  \label{lmacovtildeX}
  Assume that $\bE(T) < \infty$. Then for all $k \geq 1$ and $\alpha \geq 1$,
  \[
    |\Cov(X_0, X_k)| 
    \ll k^{- \alpha /2} + \bP ( { T} \geq [k / 4\alpha]) \, .
  \]
\end{lemma} 

\begin{proof}
  Let $k \geq 2$.
  Let $(\eps_i')_{i \geq 1}$ be an independent copy of the innovations
  $(\eps_i)_{i \geq 1}$, independent also from $g_0$.
  Define $(g_i')_{i \geq k - [k/2] + 1}$ by
  $g'_{k-[k/2]+1} = U(g_{ k- [k/2] } , \eps'_{k-[k/2]+1})$ and
  $g'_{i+1} =U(  g'_{i}, \eps'_{i+1} ) $ for $i > k- [k/2]$.
  
  Let
  \[
    X_{0,k} =  \bE_{g} \bigl( \psi (g_0, g_{1}, \ldots, g_{k-[k/2]},   (g'_{i})_{i \geq k- [k/2] +1} )
    \bigr)   \, , 
  \]
  where $\bE_{g}$ denotes the conditional expectation given $g:=(g_n)_{n \geq 0}$.
  Write
  \[
    |\Cov (X_0, X_k ) | 
    \leq  \Vert  X_k \Vert_{\infty} \Vert X_0 - X_{0,k} \Vert_1 +  | \bE  ( X_{0,k} X_k ) | 
    \, .
  \]
  Note that $\| X_k \|_{\infty} \leq | \psi |_{\infty} < \infty$.
  By Proposition~\ref{prop:boundofthedelta1first}, for any $\alpha \geq 1$, 
  \[
    \| X_0 - X_{0,k} \|_1 
    \ll k^{-\alpha/2} + \bP (T \geq [k/(4 \alpha)] )
    \, . 
  \]
  Hence it is enough to show that
  \beq \label{lmacovtildeXbut1}
    | \bE  (  X_{0,k} X_k ) |  \ll  \bP  ({T }\geq   [k/2] ) \, .
  \eeq
  With this aim, note that by the Markovian property and stationarity, 
  \[
    | \bE  (  X_{0,k} X_k ) | 
    \leq \|  X_{0,k} \|_{\infty} \| \bE (   X_k \mid g_{k- [k/2]})  \|_1
    \leq | \psi |_\infty
   \| \bE (   X_{[k/2]} \mid g_{0})  \|_1 \, .
  \]
  Recall the definition of the Markov chain \((g_n^*)_{n \geq 0}\).
  For all $n \geq 0$, let 
  \(
    X_n^* 
    = \psi ((g_k^*)_{k \geq n} )
  \).
  Since $\bE ( X_{[k/2]}^* )=0$ and  $X_{[k/2]}^*$ is independent from $g_0$, 
  \[
    \| \bE (X_{[k/2]} \mid g_{0})  \|_1 
    \leq  \| X_{[k/2]}  - X_{[k/2]}^*  \|_1 \, .
  \]
  Note now that $X_{[k/2]} \neq X_{[k/2]}^*$ only if ${ T} >[k/2]$.
  Hence 
  \[
    \| X_{[k/2]}  - X_{[k/2]}^*  \|_1 
    \leq 2 | \psi |_\infty \bP (T > [k/2] ) 
    \, , 
  \]
  which proves \eqref{lmacovtildeXbut1} and thus completes the proof of the lemma.
\end{proof}

For $n \geq 1$, let $S_n = \sum_{k=1}^{n} X_k$. From Lemma \ref{lmacovtildeX}, we get 

\begin{cor} 
  \label{cor:cov}
  Assume that $\bE(T) < \infty$. Then the limit
  \[
    c^2 = \lim_{n \to \infty} \frac{1}{n} \|S_n\|_2^2
  \]
  exists and
  \[
    c^2 = \|X_0\|_2^2 + 2 \sum_{n=1}^{\infty} \Cov (X_0, X_n)
    \, .
  \]
\end{cor}

\begin{lemma} \label{lmaFN} 
  Assume that $\bE(T) < \infty$. Then, for any $x >0$ and any $r \geq 1$, 
  \beq 
    \label{dec1maxSnproba2}
    \bP \Bigl(  \max_{k \leq n} | S_k | \geq  5 x \Bigr) 
    \ll  \frac{n}{x} \bigl( x^{-r} + \bP ( T \geq  Cx ) \bigr)
    + \Bigl( 1 + \kappa x^2/n \Bigr)^{-r/2}
    \, ,
  \eeq
  where $C$ and $\kappa$ are  constants depending on $|\psi|_\infty$ and $r$,
  and the constant involved in $\ll$ does not depend on $(n,x)$. 
\end{lemma}

\begin{proof}
  Our proof is similar to that of \cite[Thm.~6.1]{Ri00}.
  
  Let $(\eps'_n)_{n \geq 1}$ be an independent copy of the innovations $(\eps_n)_{n \geq 1}$,
  independent also of $g_0$.
  
  Fix $n \geq 1$ and $1 \leq q \leq n$.
  For $k \geq 0$, let
  \[
    X'_{k} 
    = \bE_g \bigl( \psi (g_k, g_{k+1}, \ldots, g_{k+[q/2]},
      ({\tilde g}_{i})_{i \geq  k+[q/2] +1} )  \bigr)
    \, , 
  \]
  where $\bE_g$ denotes the conditional expectation given 
  $(g_n )_{n \geq 0}$, while
  $(\tg_i)_{i \geq  k+[q/2]+1}$ is defined by
  $\tg_{k+[q/2]+1} = U (g_{k+ [q/2] } , \eps'_{k+[q/2]+1})$ and
  $\tg_{i+1} = U (\tg_i, \eps'_{i+1} ) $ for $i > k+[q/2]$.
  The function $U$ is given by \eqref{defUforMC}.
  
  Let
  \[
    S_n' = \sum_{k=1}^{n} X_k'
    \, .
  \]
  Observe that
  \[
     \max_{k \leq n} |S_k| 
     \leq \sum_{k=1}^{n} |X_k - X_{k}'| +  \max_{1 \leq k  \leq  n}  \big |  S'_k   \big | \, .
  \]
  Now, set $k_n = [n/q]$ and $U'_i =S'_{iq} - S'_{(i-1)q}$ for $1 \leq i \leq k_n$ and 
  $U'_{k_n +1} = S'_n - S'_{k_n q }$. Since all integers $j$
  are on the distance of at most $[q/2]$ from $q \bN$, we write
  \beq
    \label{dec1maxSn}
    \begin{aligned}
      \max_{k  \leq  n} |S_k|
      & \leq \sum_{k=1}^{n} |X_k - X_{k}'| 
        +  2 [q/2]  |\psi|_\infty 
      \\ & + \max_{2j \leq k_n +1}  
      \Bigl| \sum_{k=1}^j  U_{2k}' \Bigr| 
      + \max_{2j -1  \leq k_n +1}  \Bigl| \sum_{k=1}^j  U_{2k -1}' \Bigr|
      \, .
    \end{aligned}
  \eeq
  We shall now construct random variables $(U_i^*)_{1 \leq i \leq k_n +1}$ such that 
  a) $U_i^*$ has the same distribution as $U'_i$ for all $1 \leq i \leq k_n +1$,
  b) the variables $(U_{2i}^*)_{2 \leq 2i \leq k_n +1}$ are independent
     as well as the random variables $(U_{2i-1}^*)_{1 \leq 2i-1 \leq k_n +1}$ and 
  c) we can suitably control $\|U_i - U_i^*\|_1$. 

  This is done recursively as follows. 
  Let $U_{2}^* = U_2'$ and let us first construct $U_4^*$. 
  With this aim, we note that 
  \[
    X_k' = h_q   (g_k, g_{k+1}, \ldots, g_{k+[q/2]} ) 
  \]
  for some centered function $h_q$ with $|h_q|_\infty \leq |\psi|_\infty$.
  Let $g^{(2)}_{2q +[q/2]} $ be a random variable in $S$ with law $\nu$ and 
  independent from $(g_0, (\eps_k)_{k \geq 1})$ and define the Markov chain 
  $(g^{(2)}_{k})_{k \geq 2q +[q/2]}$ by: 
  \[
    g^{(2)}_{k+1} = U ( g^{(2)}_k , \eps_{k+1} ) 
    \quad \text{for } k \geq 2q +[q/2] 
    \, .
  \]
  Let
  \[
    X_k^{(2)} 
    =  h_q (g^{(2)}_k, g^{(2)}_{k+1}, \ldots, g^{(2)}_{k+[q/2]} )  
    \quad \text{for } k \geq  2q +[q/2]
  \]
  and
  \[
    U_4^* = \sum_{k= 3q +1}^{4q} X_k^{(2)} \, .
  \]
  It is clear that $U_4^* $ is independent of of $U_{2}^*$ and equal to $U'_4$ in law. 
  
  Now, for any $i \geq 3$, we define Markov chains $(g^{(i)}_k )_{ k \geq 2 (i-1) q + [q/2]}$ in the following iterative way : $g^{(i)}_{ 2 (i-1) q + [q/2]}$ is a random variable in $S$ with law $\nu$ and independent from $\big ( g_0,  (\varepsilon_k)_{k \geq 1},  \big (g^{(j)}_{2 (j-1) q + [q/2]} \big )_{2 \leq j < i } \big ) $ and we set  \[
    g^{(i)}_{k+1} = U ( g^{(i)}_k , \eps_{k+1} )
    \quad \text{for } k \geq 2 (i-1) q + [q/2]
    \, .
  \]
  Next,
  \[
    X_k^{(i)} 
    =  h_q (g^{(i)}_k, g^{(i)}_{k+1}, \ldots, g^{(i)}_{k+[q/2]} )
    \quad \text{for } k \geq  2 (i-1) q +[q/2]
  \]
  and
  \[
    U_{2i}^* = \sum_{k= (2i-1)q +1}^{2iq} X_k^{(i)}
    \, .
  \]
  It is clear that the so-constructed $(U_{2i}^*)_{2 \leq 2i \leq k_n+1}$ are independent and 
  that $U_{2i}^*$ is equal in law to $U'_{2i}$ for all $i$.
  
  By stationarity, for all $1 \leq i \leq [(k_n +1)/2]$, 
  \[
    \| U_{2i}^* - U'_{2i} \|_1 
    \leq \| U_{4}^* - U'_{4} \|_1 
    \leq \sum_{k= 3q +1}^{4q} \| X'_k - X_k^{(2)} \|_1 
    \, .
  \]
  But, by stationarity again, 
  \[
    \sum_{k= 3q +1}^{4q} \| X_k - X_k^{(2)} \|_1 
    = \sum_{k= q - [q/2] +1}^{2q - [q/2] }   \| h_q (g_k, g_{k+1}, \ldots, g_{k+[q/2]} )  
      - h_q (g^{*}_k, g^{*}_{k+1}, \ldots, g^{*}_{k+[q/2]} )  \|_1
      \, , 
  \]
  where $(g_k^*)_{k \geq 0}$ is the Markov chain defined in~\eqref{defMCstar}.
  Hence, for all $1 \leq i \leq [(k_n +1)/2]$, 
  \beq  \label{bound1diffU2i}
    \| U_{2i}^* - U'_{2i} \|_1
    \leq  2 |\psi|_\infty \sum_{k=q - [q/2] +1}^{2q - [q/2] }
    \bP ( T \geq k ) \leq 2 q  \vert \psi \vert_{\infty}   \bP ( T \geq [q/2] )   \, .
  \eeq
  
  Similarly for the odd blocks, we can construct random variables 
  $(U_{2i-1}^*)_{1 \leq 2i-1 \leq k_n +1}$ which are independent
  and such that $U_{2i -1}^*$ equals in law to $U'_{2i-1}$ for all $i$
  and
  \beq  \label{bound2diffU2i}
    \|U_{2i-1}^* - U'_{2i-1} \|_1 
    \leq 2 q  | \psi |_\infty 
    \bP ( T \geq [q/2] )
    \, .
  \eeq
  
  Overall, from~\eqref{dec1maxSn},~\eqref{bound1diffU2i} and~\eqref{bound2diffU2i},
  we deduce that for all $x >1$ and $1 \leq q \leq n$ such that $ q |\psi|_\infty \leq x $, 
  \beq
    \label{dec1maxSnproba}
    \begin{aligned}
      \bP \Bigl( \max_{k \leq  n} |S_k| \geq  5 x \Bigr) 
      \leq x^{-1} \sum_{k=0}^{n-1} \|X_k - X_{k}' \|_1 
        +  2 n x^{-1} |\psi |_\infty \bP (T \geq [q/2])  
      \\ + \bP \Bigl( \max_{2j \leq  k_n +1}  
        \Bigl| \sum_{k=1}^j  U_{2k}^* \Bigr| \geq x  \Bigr) 
      + \bP \Bigl( \max_{2j -1  \leq  k_n +1}  
        \Bigl| \sum_{k=1}^j  U_{2k -1}^* \Bigr| \geq x \Bigr)  \, .
    \end{aligned}
  \eeq
  By Proposition~\ref{prop:boundofthedelta1first},
  for all $\alpha \geq 1$, 
  \beq \label{boundofthedelta1ter}
    \| X_k - X'_{k}   \|_1 \ll q^{-\alpha/2} +   \bP  ( T \geq  [q/2] / \alpha  ) \, , 
  \eeq
  where the constant involved in $\ll$ does not depend on $k$ or $q$.
  Using that
  $ \| U_{2i}^* \|_\infty \leq q | \psi |_\infty $, we apply Bennet's inequality and derive
  \[
    \bP \Bigl( \max_{2j \leq k_n +1} 
    \Bigl| \sum_{k=1}^j U_{2k}^* \Bigr| \geq x  \Bigr)
    \leq 2 \exp \Bigl( - \frac{x}{ 2 q |\psi|_\infty}
    \log \big ( 1 + x q \vert \psi \vert_{\infty} / v_q \big ) \Bigr) \, , 
  \]
  where one can take $v_q$ any real such that 
  \[
    v_q \geq \sum_{i=1}^{  [(k_n +1)/2]} \Vert U^*_{2i } \Vert^2_2= \sum_{i=1}^{  [(k_n +1)/2]} \Vert U'_{2i } \Vert^2_2 \, .
  \]
But, by stationarity, 
  \[
    \| U'_{2i } \|_2 
    = \| S_q' \|_2 
    \leq \| S_q \|_2 + (2 |\psi|_\infty)^{1/2} 
    \sum_{k=1}^q \| X_k - X'_{k} \|^{1/2}_1
    \, .
  \]
  By Corollary~\ref{cor:cov},
  \(
    \|S_q\|_2^2 
    \ll q
  \).
  Since $n \bP ( T \geq n )  \ll 1 $, we infer that
  \[
    \sum_{k=1}^q \Vert X_k - X'_{k}   \Vert^{1/2}_1 \ll q^{1/2} \, . 
  \]
  Therefore, $\| U'_{2i} \|_2^2 \ll q $. Hence, 
  taking $v_q = n / \kappa'$ where $\kappa'$  is a sufficiently small
  positive constant not depending on $x$, $n$ and $q$, we get 
  \beq 
    \label{dec1maxSnprobabound1}
    \bP \Bigl(\max_{2j \leq k_n +1} 
      \Bigl| \sum_{k=1}^j U_{2k}^* \Bigr| \geq x  
    \Bigr) 
    \leq 2 \exp \Bigl( - \frac{x}{ 2 q |\psi|_\infty }
      \log \bigl(1 + \kappa' x q |\psi|_\infty / n \bigr) \Bigr)  \, .
  \eeq

  It follows from \eqref{dec1maxSnproba}, \eqref{boundofthedelta1ter} and
  \eqref{dec1maxSnprobabound1}, that for all $\alpha \geq 1$, $x>0$ and $1 \leq q < n$ with
  $q |\psi|_\infty \leq x $,
  \begin{multline*} 
    \bP \Bigl( \max_{k \leq n} |S_k| \geq  5 x \Bigr) 
    \ll  n x^{-1}  \bigl( 
      q^{-\alpha/2} + \bP ( T \geq  [q/2] / \alpha )
    \bigr) 
    \\ 
    + \exp \Bigl( - \frac{x}{ 2 q |\psi|_\infty }
      \log \bigl( 1 +  \kappa ' x q |\psi|_\infty /n   \bigr) \Bigr)
    \, .
  \end{multline*}
  
  Let now $r \geq 1$. Then, for $ x \in [ r |\psi|_\infty, n |\psi|_\infty / 5 ]$,
  choose $q = [ x / (r\vert \psi \vert_{\infty} )]$ and $ \alpha = 2 r$ in the previous inequality and 
  the result follows. To end the proof, note that if $x > n |\psi|_\infty / 5$, the deviation probability obviously equals zero and if $0 < x < r |\psi|_\infty$, the inequality follows easily from Markov's inequality at order $1$. 
\end{proof}

The following Rosenthal-type inequality relates $T$ 
to the moments of $S_n$.

\begin{prop} 
  \label{propRosenthal} 
  Assume  that $\bE (T) < \infty$. Then, for each $p \geq 2$, there exist $\kappa_1, \kappa_2, \kappa_3 > 0$ such that
  for all $n \geq 1$,
  \[
    \bE \Bigl(\max_{k  \leq  n} |S_k|^p \Bigr)
    \leq \kappa_1 n^{p/2} 
    + \kappa_2 n \sum_{i=1}^{ [ \kappa_3 n] } i^{p-2} \bP ( T  \geq  i  ) 
    \, .
  \]
\end{prop}

\begin{proof}
  Write
  \beq 
    \label{boundEspviaProba}
    \bE \bigl(  \max_{k  \leq n}  |S_k|^p \bigr) 
    = p  5^p \int_0^{ n \vert \psi \vert_{\infty} /5 } x^{p-1} 
    \bP \bigl( \max_{k  \leq  n}  |S_k| > 5 x \bigr)
    \, dx
    \, .
  \eeq
  Using Lemma~\ref{lmaFN} with $r = p+1$, we get that for $p \geq 2$,
  \[
    \int_{r |\psi|_\infty}^{n |\psi|_\infty / 5} 
    x^{p-1} \bP \bigl(\max_{k \leq n} |S_k| \geq  5 x \bigr) \, dx
    \ll n^{p/2} + 
    n \int_{r |\psi|_\infty}^{n |\psi|_\infty / 5} 
    x^{p-2} \bP ( T \geq  Cx ) \, dx
    \, .
  \]
  Together with \eqref{boundEspviaProba}, the above implies that for any $p \geq 2$, 
  \[
    \bE \Bigl(\max_{k \leq n} |S_k|^p \Bigr)
    \ll n^{p/2 } + 
    n \int_0^{C n |\psi|_\infty / 5} x^{p-2} \bP (T \geq x) \, dx
    \, , 
  \]
  where the constant involved in $\ll$ depends on $p$ but not on $n$.
  The result follows.
\end{proof}

\section{Proof of Theorem~\ref{ThASIPLSV}}
\label{sec:proofASIP}

\subsection{Outline}

Let $g_0, g_1, \ldots$ be the stationary Markov chain constructed
in Section~\ref{sec:RMC}. Suppose that $\psi \colon \Omega \to \bR$ is a 
H\"older continuous observable with $\int \psi \, d\bP_\Omega = 0$. Let
\[
  X_n = \psi \circ \sigma^n
  = \psi (g_n, g_{n+1}, \ldots)
  \qquad \text{and} \qquad
  S_n = \sum_{k=1}^{n} X_k
  \, .
\]
By Corollary~\ref{cor:proj}, the proof of Theorem~\ref{ThASIPLSV} reduces to 
proving ASIP with the same rates for the process $(S_n)_{n \geq 1}$.
This is the aim of this section. Our strategy is to adapt the argument in~\cite{CDM17}.

\begin{rmk}
We restrict to the case when the variance $c^2$, given by \eqref{defc2}, is positive.
The case $c^2 = 0$ requires a different approach, and it is addressed
by Remark \ref{cob}.
\end{rmk}

The Markov chain $(g_n)_{n \geq 0}$
behaves similarly to the Markov chain $(W_n)_{n \geq 0}$ on the state space
$\bN$, studied in \cite[Sec.~3.3.1]{CDM17}.
Let us briefly recall \cite[Cor.~5]{CDM17}:
For any bounded and centered function $h \colon \bN \to \bR$,
the process 
$\bigl(\sum_{k=1}^{n} h (W_k)\bigr)_{n \geq 1}$ satisfies the ASIP with rate 
$o(n^{1/p} )$, $p>2$, provided that 
$\sum_{ k \geq 1 } k^{p-2} \bP ( T \geq k ) < \infty$ where $\nu$ is the stationary distribution of $(W_n)_{ n \in \bN }$ and $T$ is the meeting time of the Markov chain.

\begin{rmk}
  By \cite[Prop.~15]{CDM17}, the condition $\sum_{ k \geq 1 } k^{p-2} \bP ( T \geq k ) < \infty$
  is sharp to get the rate $o(n^{1/p} )$ in the ASIP. 
  \end{rmk}

The strategy used in \cite{CDM17} was to adapt the method of Berkes, Liu and Wu \cite{BLW14} for functions of iid r.v.'s to functions of  Markov chains, in order  to obtain  sufficient conditions   for the ASIP  with rate $o (n^{1/p})$  in terms of an ${\mathbb L}^1$-coupling coefficient. For the Markov chain $(W_n)_{ n \in \bN }$, this ${\mathbb L}^1$-coupling condition can be obtained from the tails of the meeting time. 

The main difference between our situation and the one considered  in~\cite{CDM17} is that
$X_n$'s are functions of not only $g_n$, but the whole future $g_n, g_{n+1}, \ldots$
However, using the regularity of our observables,
we shall see that it is possible to 
approximate $X_n$ by a measurable function of a finite number of coordinates.
Then the proof in \cite{BLW14} can be adapted also to our situation, and  the rate in the ASIP is, as in \cite{CDM17}, related to the tail of the meeting time of the  chain $(g_n)_{ n \geq 0}$ (see Section~\ref{sec:MT}).

\subsection{The proof}

Let $c^2$ be given by \eqref{defc2}. From Corollaries  \ref{cor:proj} and \ref{cor:cov}, $c^2 = \lim_{n \rightarrow \infty }n^{-1} \|S_n\|^2_2 = \|X_0\|_2^2 + 2 \sum_{n=1}^{\infty} \Cov (X_0, X_n)$.
If the process $(S_n)_{n \geq 0}$ satisfies the ASIP, this has to be the variance of the limiting Brownian motion. Recall that we suppose that $c^2 > 0$.

All along the proof, we set $\beta = 1 / \gamma$ (so  $\beta >2$ since $\gamma < 1 /2$), and $\eta$ will designate a constant, which is equal either to $1$ in case of the map \eqref{eq:LSV} or to $0$ 
in case of the map \eqref{eq:HG}. 

It suffices to prove the following strong approximation: one can redefine $(S_n)_{n \geq 1}$ without changing its distribution on a  probability space (possibly richer than 
 $(\Omega, \bP_{\Omega} )$)   on which
there exists a sequence $(N_i)_{i \geq 1}$ of iid centered Gaussian r.v.'s with variance 
$c^2 $ such that for all $\kappa >1/\beta$, 
\begin{equation}
  \label{resultSAter}
  \sup_{k \leq n } \Bigl| S_n - \sum_{i=1}^k N_i \Bigr| 
  = o( n^{ 1/\beta} (\log n)^{\eta \kappa} ) \quad \text{a.s.} 
\end{equation}
The proof of \eqref{resultSAter} is divided in several steps.
Throughout, we use the notation $b_n = \lceil (\log n)/(\log 3) \rceil$
for $n \geq 2$ (so that $b_n$ is the unique integer such that $3^{b_n -1} < n \leq 3^{b_n}$),
and fix $\kappa >1/\beta$.

\medskip

\noindent \textit{Step 1.} For $\ell \geq 0$,  let 
\beq \label{defml} m_{\ell}= [ 3^{\ell/ \beta} \ell^{\eta \kappa }]
\eeq and define, for $k \geq 0$,   
\[
X_{\ell , k} =  \bE_g \big ( \psi (g_k, g_{k+1}, \ldots, g_{k+m_{\ell}},   ({\tilde g}_{i})_{i \geq  k+m_{\ell} +1} )  \big ) \, , 
\]
where $\bE_g$ denotes the conditional expectation given $g:=(g_n )_{n \geq 0}$. 
Here $({\tilde g}_{i})_{i \geq  k+m_{\ell} +1} $ is defined as follows: ${\tilde g}_{k+m_{\ell} +1} = U (g_{k+m_{\ell} } , \eps'_{k+m_{\ell} +1})$ and 
${\tilde g}_{i+1} = U (  {\tilde g}_{i}, \eps'_{i+1} ) $ for any $i > k+m_{\ell} $, where $(\eps_i')_{i \geq 1}$ is an independent copy of $(\eps_i)_{i \geq 1}$, independent of $g_0$, and $U$ is given by \eqref{defUforMC}.  Note that the $X_{\ell , k}$'s are centered. 
Define
\[
 { W}_{\ell, i} = \sum_{k=1+3^{\ell-1}}^{ i  +  3^{\ell-1}}  X_k  \,  \text{ , }  \,  { \overline W}_{\ell, i} = \sum_{k=1+3^{\ell-1}}^{ i  +  3^{\ell-1}}  {X}_{\ell , k }  \,  \text{ and }  \, 
 { W}'_{\ell, i}  =  { W}_{\ell, i}  -  { \overline W}_{\ell, i}  \, .
 \]
The fist step is to prove that 
\begin{equation} \label{stepone}
\sum_{\ell=1}^{ b_n -1}  { W}'_{\ell, 3^{\ell} - 3^{\ell-1} }   +  { W}'_{b_n,  n - 3^{b_n-1}  }   =o( n^{ 1/\beta} (\log n)^{\eta \kappa} )  \quad  a.s. 
\end{equation}
This will hold provided that for all $\eps>0$, 
\begin{equation}\label{BC}
\sum_{j \geq 1}  {\mathbb P} \left (  \sum_{\ell=1}^j \sum_{k= 3^{\ell -1} +1}^{3^\ell}  |X_k-X_{\ell , k}| > \eps 3^{j/\beta} j^{\eta \kappa }  \right) < \infty \, .
\end{equation}
By Proposition~\ref{prop:boundofthedelta1first}, for all $k \geq 0$, $\ell \geq 1$ and $r \geq 1$, 
\beq \label{boundofthedelta1first}
  \| X_k - X_{\ell, k} \|_1 \ll m^{-r/2}_{\ell} +  \bP  ( T \geq  [m_{\ell}/r] ) \, , 
\eeq
where the constant involved in $\ll $ does not depend on $k$ and $\ell$. By Markov inequality at order $1$,
for all $\eps >0$ and $r \geq 1$,
\begin{multline*}
\sum_{j \geq 1}  {\mathbb P} \left (  \sum_{\ell=1}^j \sum_{k= 3^{\ell -1} +1}^{3^\ell}  |X_k-X_{\ell , k}| > \eps 3^{j/\beta} j^{\eta \kappa} \right) 
\\
\ll \sum_{j \geq 1}  \frac{ 1 }{ \eps 3^{j/\beta} j^{\eta \kappa }} \sum_{\ell=1}^j  3^{\ell }  m^{-r/2}_{\ell}  + \sum_{j \geq 1}  \frac{ 1 }{ \eps 3^{j/\beta} j^{\eta \kappa}} 
\sum_{\ell=1}^j  3^{\ell } 
\bP  ( T \geq  [m_{\ell}/r] ) \, . \end{multline*}
Taking into account the fact that $m_{\ell}= [ 3^{\ell/ \beta} \ell^{\eta \kappa }]$, the first term in the right-hand side is finite provided we take $r > 
2 ( \beta -1)$ whereas, by a change of variables,  we have, for any $r \geq 1$, 
\beq  \label{evident}
\sum_{j \geq 1}  \frac{1}{ 3^{ j/\beta} j^{\eta \kappa}} \sum_{\ell=1}^j  3^{\ell }  \bP  ( T \geq  [m_{\ell}/r] ) \leq C \sum_{n \geq 2} \frac{n^{\beta -2}}{ (\log n)^{\eta \kappa \beta  }}  \bP  ( T \geq  n)  \, .
\eeq
where $C$ is a constant depending on $r$, $\beta$, $\kappa$ and $\eta$.  In case of the map \eqref{eq:LSV}, $\eta =1$ and  the series above converge  iff $\bE (\tpsi_{\beta, \kappa \beta} (T) )  < \infty$, which holds by Lemma \ref{lem:nngh}(a) and the fact that  $\kappa \beta >1 $. Now in case of the map \eqref{eq:HG}, $\eta =0$ and then, again from  Lemma \ref{lem:nngh},  the series above converges  since $\bE (T^{\beta -1} ) < \infty$. It follows that \eqref{BC} is satisfied and then \eqref{stepone} holds. 

This completes the proof of step 1.

\medskip

\noindent \textit{Step 2.}
Let
\beq \label{deftildeXkl}
  {\tilde X}_{\ell , k} = \bE ( X_{\ell , k}  | \eps_{k - m_{\ell}} , \ldots,  \eps_{k + m_{\ell}}  )
  \, .
\eeq
Let 
\(
  { \widetilde W}_{\ell, i} = \sum_{k=1+3^{\ell-1}}^{ i  +  3^{\ell-1}}  {\tilde X}_{\ell, k }
\)
and
\(
 { W}''_{\ell, i}  =   { \overline W}_{\ell, i} -{ \widetilde W}_{\ell, i }
\).
The second step consists of proving that
\begin{equation} \label{steptwo}
\sum_{\ell=1}^{ b_n -1}  { W}''_{\ell, 3^{\ell} - 3^{\ell-1} }  +  { W}''_{b_n,  n - 3^{b_n-1}  } = o( n^{1/\beta} (\log n)^{\eta \kappa}) \quad  a.s.   \, .
\end{equation}

Clearly, \eqref{steptwo} will follow from  the Kronecker lemma, if one can prove that
\beq \label{krosteptwo}
\sum_{\ell \geq 1}   \frac{1}{ 3^{ \ell/\beta} \ell^{\eta \kappa}} \sum_{k= 3^{\ell-1}+1}^{3^\ell} \Vert X_{ \ell , k }  - {\tilde X}_{\ell , k}  \Vert_1  < \infty \, .   
\eeq
We claim that \beq \label{boundofthedelta1}
\Vert X_{\ell, k}  - {\tilde X}_{\ell , k}  \Vert_1 \leq  2 \vert \psi \vert_{\infty}  \bP(T \geq  m_{\ell}) \, .
\eeq
Then, using \eqref{boundofthedelta1},
\[
\sum_{\ell \geq 1}   \frac{1}{ 3^{ \ell/\beta} \ell^{\eta \kappa}} \sum_{k= 3^{\ell-1}+1}^{3^\ell} \Vert X_{ \ell , k }  - {\tilde X}_{\ell , k}  \Vert_1 \leq 2 \vert \psi \vert_{\infty}
\sum_{\ell \geq 1}   \frac{3^{\ell}}{ 3^{ \ell/\beta} \ell^{\eta \kappa}}  \bP (T \geq  m_{\ell}) \, .  
\]
Therefore \eqref{krosteptwo} holds by using \eqref{evident} and Lemma \ref{lem:nngh} (as quoted right after \eqref{evident}). 

It remains to prove the claim  \eqref{boundofthedelta1}. This follows closely the proof of 
\cite[Lem.~24]{CDM17}. Indeed, we can write
\[
X_{\ell , k} =   h_{\ell} (g_k, g_{k+1}, \ldots, g_{k+m_{\ell}}  )  \, , 
\]
where $h_{\ell} $ is a measurable function such that $\vert h_{\ell} \vert_{\infty}  \leq \vert  \psi \vert_{\infty} $ and $\bP_{\Omega} ( h_{\ell} ) = 0$. Hence 
\[
 X_{\ell, k }  - {\tilde X}_{\ell , k}  = h_{\ell}   (g_k, g_{k+1}, \ldots, g_{k+m_{\ell}} )   - \bE  \big (  h_{\ell}   (g_k, g_{k+1}, \ldots, g_{k+m_{\ell}} )  | \eps_{k - m_{\ell}} , \ldots,  \eps_{k + m_{\ell}} \big )   \, .
\] 
Recall that for all $k \geq 1$, $g_k = U ( g_{k-1} , \eps_k) $ where $U$ is a measurable function from $S \times {\mathcal A}$ to $S$. 
For any $i \geq 1$, let then $U_{i} $ be the function from $S \times {\mathcal A}^{\otimes i}$ to $S$ defined in the following iterative way: 
\[
U_1 = U \quad \text{and} \quad U_i ( a, x_1, x_2, \ldots, x_i) = U_{i-1}  \big ( U (a, x_1), x_2, \ldots, x_i  \big ) \, , \,  i \geq 2 \, . 
\]
Then for all $i \geq 0$ and $k \geq m_{\ell } +1$, 
\[
g_{k+i} = U_{i+m_{\ell }+1} ( g_{k -m_{\ell}-1} ,  \eps_{k-m_{\ell}}, \ldots, \eps_{k+i})  \, .
\]
Hence,
\begin{multline*}
h_{\ell}   (g_k, g_{k+1}, \ldots, g_{k+m_{\ell}} )  \\
= h_{\ell} \Big ( U_{m_{\ell }+1} ( g_{k -m_{\ell}-1} ,  \eps_{k-m_{\ell}}, \ldots, \eps_{k})  , \ldots ,  U_{2 m_{\ell }+1} ( g_{k -m_{\ell}-1} ,  \eps_{k-m_{\ell}}, \ldots, \eps_{k+m_{\ell }})   \Big ) \\
 =: H_{\ell, m_{\ell}} \big (  g_{k -m_{\ell}-1} ,  \eps_{k-m_{\ell}}, \ldots ,  \eps_{k+m_{\ell }}   \big ) \, .
\end{multline*}
Let now $(  \eps'_k )_{k \geq 1}$ be an independent copy of $( \eps_k )_{k \geq 1}$, independent of $g_0$. Let \(g_0'\) be a random variable in \(S\) with distribution $\nu$ and 
independent from $( g_0, ( \eps_k )_{k \geq 1}, ( \eps'_k )_{k \geq 1})$. Define a Markov chain
$(g'_n )_{n \geq 0}$ by
\[
  g'_{n+1}= U(g_n', \eps'_{n+1})
  \ \text{ for } \ n \geq 0
 \,  .
\]
Denoting $V_{k, m_{\ell}} = (g_0, \eps_1, \ldots, \eps_{k+m_{\ell }})$ and $ \bE_{V_{k, m_{\ell}}}  ( \cdot ) =  \bE ( \cdot | V_{k, m_{\ell}})$, we have
\begin{multline*}
 X_{\ell, k }  - {\tilde X}_{\ell , k}  = \bE_{V_{k, m_{\ell}}}  \Big ( H_{\ell, m_{\ell}} \big (  g_{k -m_{\ell}-1} ,  \eps_{k-m_{\ell}}, \ldots ,  \eps_{k+m_{\ell }}   \big ) \Big )
 \\  - 
 \bE_{V_{k, m_{\ell}}}  \Big ( H_{\ell, m_{\ell}} \big (  g'_{k -m_{\ell}-1} ,  \eps_{k-m_{\ell}}, \ldots ,  \eps_{k+m_{\ell }}   \big ) \Big )\, .
\end{multline*}
Hence, using the stationarity,
\begin{multline*}
\Vert  X_{\ell, k }  - {\tilde X}_{\ell , k} \Vert_1  \leq \Vert H_{\ell, m_{\ell}} \big (  g_{k -m_{\ell}-1} ,  \eps_{k-m_{\ell}}, \ldots ,  \eps_{k+m_{\ell }}   \big ) -
 H_{\ell, m_{\ell}} \big (  g'_{k -m_{\ell}-1} ,  \eps_{k-m_{\ell}}, \ldots ,  \eps_{k+m_{\ell }}   \big ) \Vert_1 \\
= \Vert H_{\ell, m_{\ell}} \big (  g_{0} ,  \eps_{1}, \ldots ,  \eps_{2m_{\ell } +1}   \big ) -H_{\ell, m_{\ell}} \big (  g'_{0} ,  \eps_{1}, \ldots ,  \eps_{2m_{\ell } +1}   \big ) \Vert_1 \, .
\end{multline*}
Let $(g^*_n )_{n \geq 0}$ be the Markov chain in the definition of the meeting time,
see \eqref{defMCstar}. Then
\begin{multline*}
\Vert  X_{ \ell , k }  - {\tilde X}_{\ell , k} \Vert_1  \leq  \Vert H_{\ell, m_{\ell}} \big (  g_{0} ,  \eps_{1}, \ldots ,  \eps_{2m_{\ell } +1}   \big ) -H_{\ell, m_{\ell}} \big (  g^*_{0} ,  \eps_{1}, \ldots ,  \eps_{2m_{\ell } +1}   \big ) \Vert_1 \\
= \Vert  h_{\ell} (g_{m_{\ell} +1}, g_{m_{\ell} +2}, \ldots, g_{2m_{\ell} +1}  )  - h_{\ell} (g^*_{m_{\ell} +1}, g^*_{m_{\ell} +2}, \ldots, g^*_{2m_{\ell} +1}  )  \Vert_1 \, .
\end{multline*}
Recall that for every $k \geq  T $, $g_k = g_k^*$. Therefore
\[
\Vert  X_{ \ell , k }  - {\tilde X}_{\ell , k} \Vert_1 \leq 2 \vert h_{\ell} \vert_{\infty}  \bP(T \geq  m_{\ell} ) \, , 
\]
proving  \eqref{boundofthedelta1}. This ends the proof of step 2.

\medskip

\noindent \textit{Step 3.}  Setting ${\tilde S}_n := \sum_{\ell=1}^{ b_n -1} { \widetilde W}_{\ell, 3^{\ell} - 3^{\ell-1} }    +  { \widetilde W}_{b_n,  n - 3^{b_n-1}  } $,  the rest of the proof consists in showing that,  enlarging the underlying probability space if necessary, there exists a sequence $(N_i)_{i \geq 1}$ of iid centered Gaussian r.v.'s with variance 
$c^2 $ such that 
\begin{equation} \label{resultSAter*}
\sup_{k \leq n } \Big | {\tilde S}_k - \sum_{i=1}^k N_i \Big  | = o( n^{ \gamma} (\log n)^{\eta \kappa} ) \quad   a.s. 
\end{equation}
This can be achieved using the method of \cite{BLW14}.
Indeed the constructed $ {\tilde X}_{\ell , k} $ can be rewritten as 
\[
 {\tilde X}_{\ell , k}  := G_{\ell} (  \eps_{k - m_{\ell}} , \ldots,  \eps_{k + m_{\ell}}  ) \, , 
\]
where $G_{\ell}$ is a measurable function. So $ {\tilde X}_{\ell , k}$ is a measurable function of  $(  \eps_{k - m_{\ell}} , \ldots,  \eps_{k + m_{\ell}}  )$ instead of $ (  \eps_{k - m_{\ell}} , \ldots,  \eps_{k}  )$ as in \cite{BLW14}. However, this difference can be handled by only minor adjustments, mainly taking  $2 m_{\ell}$ instead of  $m_{\ell}$ in \cite{BLW14}. More precisely, the blocks $B_{\ell,j}$ in \cite{BLW14} can be defined as follows:
for 
$\ell \geq k_0 := \inf \{ k \geq 1: m_k \leq 4^{-1} 3^{k-2} \}$
and $j=1, \ldots, q_\ell : =   \lceil 3^{\ell-2} /m_\ell \rceil -2 $,
\[
B_{\ell,j} =  \sum_{i=1+ (6j-1) m_\ell }^{ (6 j +5) m_\ell} {\tilde X}_{\ell,i +m_\ell +3^{\ell-1}} \, . 
\]
Define, for $j \geq 1$,
\[
{\mathcal J}_{\ell,j} = \{ 3^{\ell-1} +  (6j-1) m_\ell +  k , k=1,2, \ldots, 2m_\ell\}  \, , 
\]
\[
{\mathbf U}_{\ell,j} = (\eps_{i}, i \in {\mathcal J}_{\ell,j}  ) \,  \text{ and } \, {\bf U} = ({\bf U}_{\ell,j}, j =1, \ldots, q_\ell+1)_{\ell=k_0}^{\infty} \, .
\]
Then
\begin{multline*}
B_{\ell,j} =  \sum_{i=1+ (6j-1) m_\ell }^{ (6 j +1) m_\ell} {\tilde X}_{\ell,i +m_\ell +3^{\ell-1}}  +  \sum_{i=1+ (6j +1) m_\ell }^{ (6 j +3) m_\ell} {\tilde X}_{\ell,i +m_\ell +3^{\ell-1}}  +  \sum_{i=1+ (6j+3) m_\ell }^{ (6 j +5) m_\ell} {\tilde X}_{\ell,i +m_\ell +3^{\ell-1}}   \\ := H_\ell \big ( {\mathbf U}_{\ell,j}, \{\eps_{i + 3^{\ell-1}}\}_{ 1+(6j +1) m_\ell  \leq i  \leq (6j +5) m_\ell} , {\mathbf U}_{\ell,j+1} \big ) 
\end{multline*}
On the set $\{{\bf U} = {\bf u} \}$,  $(B_{\ell,j} ( {\bf u} ))_{j=1, \ldots, q_\ell}$ are then independent between them. Then, following \cite{BLW14}, we use Sakhanenko's strong approximation \cite{Sa06} to get a bound for the approximation error between ${\tilde S}_n ( {\bf u} )$ and a Wiener process with variance depending on $ {\bf u}$. To get the unconditional ASIP, we use the arguments given in \cite[step~3.4]{BLW14}.  So, as it is summarized in \cite[Prop.~21]{CDM17}, we infer that \eqref{resultSAter*} will follow if one can prove that  there exists $ r \in (2, \infty )$ such that 
\beq \label{condtoapplySakhanenko}
\sum_{\ell \geq k_0} \frac { 3^{\ell} }{  3^{\ell r /\beta} \ell^{  \eta \kappa r } m_\ell}   \bE \Big (  \max_{1 \leq k  \leq  6m_\ell }  \big |   {\widetilde W}_{\ell , k }    \big |^r \Big ) < \infty \, ,
\eeq
and
\beq \label{cond2v_k}
3^\ell (  \nu_\ell^{1/2}  - c)^2 = o(3^{2\ell/ \beta} \ell^{ 2  \eta \kappa  }  (\log \ell)^{-1})  \, , \, \mbox{ as $\ell \rightarrow \infty$}\, ,  
\eeq
where 
\beq \label{defnuk}
\nu_\ell = (2m_\ell)^{-1}  \big \{ \bE ({\widetilde W}^2_{\ell, 2m_\ell} )  + 2 \bE ({\widetilde W}_{\ell, 2m_\ell} ({\widetilde W}_{\ell, 4m_\ell}  -{\widetilde W}_{\ell, 2m_\ell} ))\big \} \, .
\eeq
To end the proof, it remains to prove the two conditions above. 
We start with \eqref{condtoapplySakhanenko}. Note first that for all $r \geq 1$, 
\[
\Big \Vert \max_{1 \leq k  \leq  6m_\ell }  \big |  { W}_{ k }  -  {\widetilde W}_{\ell , k }    \big |  \Big \Vert_r \leq  \sum_{k=1+3^{\ell-1}}^{ 6 m_{\ell}  +  3^{\ell-1}} \Vert  X_k - {\tilde X}_{\ell, k } \Vert_r  \, .
\]
Using that $\Vert X_k \Vert_{\infty} \leq \vert \psi \vert_{\infty}$ and $ \Vert {\tilde X}_{\ell, k } \Vert_{\infty} \leq 2 \vert \psi \vert_{\infty}$, we get 
\[
\Big \Vert \max_{1 \leq k  \leq  6m_\ell }  \big |  { W}_{ k }  -  {\widetilde W}_{\ell , k }    \big |  \Big \Vert_r \leq  ( 3 \vert \psi \vert_{\infty})^{(r-1)/r}  \sum_{k=1+3^{\ell-1}}^{ 6 m_{\ell}  +  3^{\ell-1}} \big (  \Vert  X_k - { X}_{\ell, k } \Vert^{1/r}_1 + \Vert  X_{\ell, k} - {\tilde X}_{\ell, k } \Vert^{1/r}_1  \big )  \, .
\]
But according to \eqref{boundofthedelta1first} and \eqref{boundofthedelta1}, for all $\alpha \geq 1$, 
\beq \label{boundofthedelta1bis}
\Vert X_{k}  - {\tilde X}_{\ell , k}  \Vert_1 \ll m^{-\alpha/2}_{\ell} +  \bP  ( T \geq  [m_{\ell}/\alpha] )  \, ,
\eeq
where the constant involved in $\ll$ does not depend on $k$ and $\ell$.
Therefore, for all $r \geq 1$ and $\alpha \geq 1$, 
\begin{multline*}
\sum_{\ell \geq k_0} \frac { 3^{\ell} }{  3^{\ell r /\beta} \ell^{  \eta \kappa r } m_\ell}   \bE \Big (  \max_{1 \leq k  \leq  6m_\ell }  \big | W_k -   {\widetilde W}_{\ell , k }    \big |^r \Big )
 \ll \sum_{\ell \geq k_0} \frac { 3^{\ell}  m^r _\ell}{  3^{\ell r /\beta} \ell^{  \eta \kappa  r } m_\ell} \big ( m^{-\alpha/2}_{\ell} +  \bP  ( T \geq  [m_{\ell}/\alpha] ) \big ) \\
  \ll \sum_{\ell \geq k_0} \frac { 3^{\ell (\beta -1)/\beta }  }{   \ell^{  \eta \kappa  }}  3^{- \alpha \ell / ( 2 \beta ) }  \ell^{ - \alpha \eta \kappa /2} + \sum_{\ell \geq k_0} \frac { 3^{\ell (\beta -1)/\beta }  }{   \ell^{  \eta \kappa  }}  \bP  ( T \geq  3^{\ell / \beta } \ell^{ \eta \kappa}/\alpha )  \, . 
\end{multline*}
The first term in the right-hand side is finite  provided that we take $\alpha > 
2 ( \beta -1)$ whereas, the second series  converge for any $\alpha \geq 1$,  by using once again \eqref{evident} and Lemma \ref{lem:nngh}. Therefore, to prove \eqref{condtoapplySakhanenko}, it suffices to show 
that there exists $ r \in ]2, \infty [$ such that 
\beq \label{condtoapplySakhanenko2}
\sum_{\ell \geq k_0} \frac { 3^{\ell} }{  3^{\ell r /\beta} \ell^{ \eta \kappa r } m_\ell}   \bE \Big (  \max_{1 \leq k  \leq  6m_\ell }  \big |  W_k   \big |^r \Big ) < \infty \, .
\eeq

By Lemma \ref{lem:nngh}, $\bE (T) < \infty$ since $\beta >2$ for both maps. Using stationarity and Proposition \ref{propRosenthal}, we get that for any $r \geq 2$,
\begin{multline*} 
\sum_{\ell \geq k_0} \frac { 3^{\ell} }{  3^{\ell r /\beta} \ell^{  \eta \kappa r } m_\ell}   \bE \Big (  \max_{1 \leq k  \leq  6m_\ell }  \big |  W_k   \big |^r \Big )  \\  \ll
\sum_{\ell \geq k_0} \frac { 3^{\ell} }{  3^{\ell r /\beta} \ell^{ \eta \kappa r } }  m_\ell^{r/2-1}+
\sum_{\ell \geq k_0} \frac { 3^{\ell} }{  3^{\ell r /\beta} \ell^{ \eta \kappa r } }   \sum_{i=1}^{ [ 6 \kappa_3 m_{\ell}] } i^{r-2} \bP ( T  \geq  i  ) \, .
\end{multline*} 
Since $m_{\ell} = [3^{\ell/\beta} \ell^{ \eta \kappa}]$, the first term of the right-hand side is finite provided that we take $ r > 2 ( \beta -1)$. To control the second term, we note that for any $r > \beta$, by a change of variables, 
\[
\sum_{\ell \geq k_0} \frac { 3^{\ell} }{  3^{\ell r /\beta} \ell^{ (1+ \eta) r } }   \sum_{i=1}^{ [ 6 \kappa_3 m_{\ell}] } i^{r-2} \bP ( T  \geq  i  )  \ll 
\sum_{i \geq 2 } \frac{i^{\beta-2} }{(\log i )^{ \eta \kappa r} } \bP ( T  \geq  i  )  
\]
which is finite by Lemma \ref{lem:nngh} as it was quoted after \eqref{evident}. So, provided that we take $ r > 2 ( \beta -1)$, since $\beta >2$,  \eqref{condtoapplySakhanenko2} holds (and then \eqref{condtoapplySakhanenko}). 



\smallskip

We turn now to the proof of \eqref{cond2v_k}. Proceeding as to get the relation \cite[(66)]{CDM17}, we have
\[
\nu_{\ell} =  {\tilde c}_{\ell, 0}  + 2 \sum_{k =1}^{2m_\ell }  {\tilde c}_{\ell,k}  \, , 
\]
where, for any $i \geq 0$, 
\[
{\tilde c}_{\ell,i}= \Cov ({\tilde X}_{\ell, m_\ell+1} , {\tilde X}_{\ell,i+m_\ell+1} ) \, .
\]
Note also that since $c^2$ is assumed to be positive, to prove \eqref{cond2v_k}, it suffices to prove that 
\beq \label{cond2v_kbis}
3^\ell (  \nu_\ell - c^2)^2 = o(3^{2\ell/ \beta} \ell^{ 2\eta \kappa }  (\log \ell)^{-1})  \, , \, \mbox{ as $\ell \rightarrow \infty$}\, . 
\eeq
To show that \eqref{cond2v_kbis} is satisfied, we first note that, by stationarity, for all $i \geq 0$, 
\begin{multline*}
\big |  {\tilde c}_{\ell,i} - \Cov  ( X_0, X_i ) \big |  = \big | \Cov ({\tilde X}_{\ell, m_\ell+1}  - X_{m_\ell+1}, {\tilde X}_{\ell,i+m_\ell+1} )   +  \Cov  ( X_{m_\ell+1} , 
{\tilde X}_{\ell,i+m_\ell+1}  - X_{i+m_\ell+1} ) \big | \\
\leq 2 \vert \psi \vert_{\infty} \big ( \Vert {\tilde X}_{\ell, m_\ell+1}  - X_{m_\ell+1} \Vert_1+ \Vert  {\tilde X}_{\ell, i+m_\ell+1}  - X_{i+m_\ell+1} \Vert_1 \big ) \, .
\end{multline*}
Let $\alpha \geq 1$. Then, according to \eqref{boundofthedelta1bis}, for all $i \geq 0$,
\[
\big |  {\tilde c}_{\ell,i} - \Cov  ( X_0, X_i ) \big |   \ll m^{-\alpha/2}_{\ell} +  \bP  ( T \geq  [m_{\ell}/\alpha] )  \, .
\]
It follows that
\[
| \nu_{\ell}  -  c^2 | \ll m^{1-\alpha/2}_{\ell} + m_{\ell} \bP  ( T \geq  [m_{\ell}/\alpha] )  + 2 \sum_{i > 2m_{\ell}}  \big | \Cov  ( X_0, X_i ) \big | \, .
\]

Recall that $\beta > 2$. By Lemma~\ref{lem:nngh} (since $\kappa \beta >1$), 
\[
\bP ( T \geq n )  = o \big (  ( \log n )^{\eta \kappa \beta } n ^{1-\beta}\big ) \, ,  \, \mbox{ as $n \rightarrow \infty$} \, .
\]
Using, in addition, Lemma~\ref{lmacovtildeX}, we derive  that for all $\alpha \geq 1$, 
\[
| \nu_{\ell}  -  c^2 | \ll 3^{ \ell (2-\alpha)/ ( 2 \beta) }  \ell^{\eta \kappa(2-\alpha)/2 }+ o \big ( 3^{\ell  (2 - \beta) /{\beta} } \ell^{ 2\eta \kappa} \big )   \, ,
\]
proving \eqref{cond2v_kbis} (and then \eqref{cond2v_k}) using the fact that $\beta >2$ and taking $\alpha \geq 2\beta  - 2$.  This ends the proof of Theorem \ref{ThASIPLSV} 
when $c^2 >0$.

\subsection{Extension to other observables} \label{piecewise-holder}
As already mentioned in Remark \ref{rmk:1-6}, it is possible to relax the H\"older continuity assumption. 
For instance, if $m\ge 1$ is an 
integer, assume that  $\varphi$ is H\"older on the interior of $Y_{a_0\cdots a_{m-1}}$ for every  $a_0,\ldots, a_{m-1}\in \alpha$. Denote by $\alpha(a_0,\ldots, a_{m-1})$ the corresponding H\"older exponent 
and by $\vert \varphi \vert_{\alpha(a_0,\ldots, a_{m-1})}$ the corresponding H\"older norm. Assume further that 
$\alpha^*:=\inf_{a_0,\ldots , a_{m-1}\in \alpha}\alpha(a_0,\ldots, a_{m-1})>0$ and that 
$ \vert \varphi \vert_{\alpha*}:=\sup_{a_0,\ldots , a_{m-1}\in \alpha} \vert \varphi \vert_{\alpha(a_0,\ldots, a_{m-1})}<\infty$. 
Under the above assumptions, the conclusion of  Theorem \ref{ThASIPLSV} holds. 
\medskip

Let us briefly give the arguments explaining why such an extension is possible. We just give the necessary arguments to prove the estimate \eqref{boundofthedelta1first} (or more generally Proposition~\ref{prop:boundofthedelta1first}). Similar arguments may be used at each place where the H\"older property has been used  to get similar estimates as \eqref{boundofthedelta1first}.  To do so one has to bound

\begin{equation}\label{holder-bound}
  |\psi(g_0, \ldots, g_n, g_{n+1}, \ldots) - \psi(g_0, \ldots, g_n, ({\tilde g}_{k})_{k \geq n+1}, \ldots)|
\end{equation} 

 If $\#\{k \leq n \colon g_k  \in S_0\}<m$  we bound \eqref{holder-bound} by $2|\varphi|_\infty$ .

\medskip

Assume now  that $\#\{k \leq n \colon g_k  \in S_0\}\ge m$. Set $g_0=(w_0,\ell_0)$. Assume that we can write that 
$w_0=ww'$ with $h(w)=\ell_0$ and $w$ may be an emptyword (in which case $\ell_0=0$). Hence, $\pi(g_0, \ldots, g_n, g_{n+1}, \ldots) $ and 
$\pi(g_0, \ldots, g_n, ({\tilde g}_{k})_{k \geq n+1}, \ldots) $ belong to $Y$ and even, since $\#\{k \leq n \colon g_k  \in S_0\}\ge m$, 
to some $Y_{a_0\cdots a_{m-1}}$ (on which $\varphi$ is H\"older). In particular one may bound 
\eqref{holder-bound} by $\vert \varphi \vert_{\alpha*} \lambda^{- \alpha* \#\{k \leq n \colon g_k  \in S_0\}}$.

If $w_0$ cannot be written as above then, $\pi(g_0, \ldots, g_n, g_{n+1}, \ldots) $ and 
$\pi(g_0, \ldots, g_n, ({\tilde g}_{k})_{k \geq n+1}, \ldots) $ belongs to $[0,1/2)$ and we infer a 
similar bound.

So at the end, there exists $C>0$ depending on $|\varphi|_\infty$ and $ \vert \varphi \vert_{\alpha*}$, such that 
\begin{equation*}
  |\psi(g_0, \ldots, g_n, g_{n+1}, \ldots) - \psi(g_0, \ldots, g_n, ({\tilde g}_{k})_{k \geq n+1}, \ldots)|
\le C\theta^{ \#\{k \leq n \colon g_k  \in S_0\}-m}\, .
\end{equation*} 
The end of the proof of Proposition~\ref{prop:boundofthedelta1first} remains unchanged.

\section{Nonuniformly expanding  dynamical systems}
\label{sec:NUE}

We stated and proved Theorem~\ref{ThASIPLSV} for two particular families of maps.
In this section we extend our result to the class of nonuniformly expanding systems
which admit inducing schemes as in Young~\cite{Y99} with polynomially
decaying tails of return times.

\subsection{Nonuniformly expanding maps}

Let $X$ be a complete bounded separable metric space with the Borel $\sigma$-algebra.
Suppose that $f \colon X \to X$ is a measurable transformation which
admits an inducing scheme consisting of:
\begin{itemize}
  \item a closed subset $Y$ of $X$ with a \emph{reference} probability measure $m$ on $Y$;
  \item a finite or countable partition $\alpha$ of $Y$ (up to a zero measure set)
    with $m(a) > 0$ for all $a \in \alpha$;
  \item an integrable \emph{return time} function $\tau \colon Y \to \{1,2,\ldots\}$ which is constant
    on each $a \in \alpha$ with value $\tau(a)$ and
    $f^{\tau(a)}(y) \in Y$ for all $y \in a$, $a \in \alpha$.
    (We do not require that $\tau$ is the \emph{first} return time to $Y$.)
\end{itemize}

Define $F \colon Y \to Y$, $F(y) = f^{\tau(y)}(y)$.
We assume that there are constants $\kappa > 1$, $K > 0$ and $\eta \in (0,1]$
such that for each $a \in \alpha$ and all $x,y \in a$:
\begin{itemize}
  \item $F$ restricts to a (measure-theoretic) bijection from $a$ to $Y$,
    nonsingular with respect to the measure $m$;
  \item $d(F(x), F(y)) \geq \kappa d(x,y)$;
  \item $d(f^k(x), f^k(y)) \leq K d(F(x), F(y))$
    for all $0 \leq k \leq \tau(a)$;
  \item the inverse Jacobian $\zeta_a = \frac{dm}{dm \circ F}$ of
    the restriction $F \colon a \to Y$ satisfies
    \[
      \bigl| \log |\zeta_a(x)| - \log |\zeta_a(y)| \bigr|
      \leq K d(F(x), F(y))^\eta
      .
    \]
\end{itemize}
The map $f$ as above is said to be nonuniformly expanding.
It is standard \cite[Cor. p.~199]{AD}, \cite[Proof of Thm.~1]{Y99} that there is
a unique absolutely continuous $F$-invariant probability measure
$\mu_Y$ on $Y$ with $\frac1c \le d\mu_Y / dm\le c$ for some $ c>0$,
and the corresponding $f$-invariant probability measure $\mu$ on $X$.

We make an additional assumption, which is not part of the usual definition
of nonuniformly expanding maps,
but is straightforward to verify in examples.
Denote by $\cA$ the set of all finite words in the
alphabet $\alpha$ (not including the empty word) and set 
$Y_w:= \cap_{k=0}^n F^{-k}(a_k)$ for $w=a_0\cdots a_n $ in $\cA$. We assume that
\begin{equation}\label{closure}
  m(Y_w)=m(\bar Y_w)\qquad \mbox{ for every $w\in \cA$.}
\end{equation}

We say that the return times of $f$ have:
\begin{itemize}
  \item a weak polynomial moment of order $\beta \geq 1$, if $m(\tau  \geq n) \ll n^{-\beta}$;
  \item a strong polynomial moment of order $\beta \geq 1$, if $\int \tau^\beta \, dm < \infty$.
\end{itemize}

\begin{rmk}
  Intermittent maps~\eqref{eq:LSV} and~\eqref{eq:HG} 
  are nonuniformly expanding. Their return times have respective weak and strong moments
  of order $\beta = 1/\gamma$.

  More generally, our results apply to nonuniformly expanding and nonuniformly
  hyperbolic dynamical systems which can be modelled by Young towers~\cite{Y98,Y99}.
  A notable example with polynomial return times is
  the class of non-Markov maps with indifferent fixed points in \cite[Sec.~7]{Y99}.
  (C.f.\ AFN maps in Zweim\"uller~\cite{Z98}.)
\end{rmk}


\subsection{Rates in the ASIP}
\label{sec:R-ASIP_NUE}

Suppose that $\varphi \colon X \to \bR$ is a H\"older continuous observable such that $\mu(\varphi)=0$. 
Let
$S_n(\varphi) = \sum_{k=0}^{n-1} \varphi \circ f^k$ be the corresponding random
process, defined on the probability space $(X, \mu)$. Assume in addition that the return times of $f$ have a polynomial moment of order $\beta >2$ (weak or strong). Let
\beq
  \label{eq:c2nomix}
  c^2 = \lim_{n \rightarrow \infty } \frac{1}{n} \int |S_n (\varphi )|^2 \, d\mu
  \, .
\eeq

\begin{rmk}
  The limit above exists by e.g.\ \cite[Cor.~2.12]{KKM}.
  In case of summable correlations, $c^2$ can be computed by
  the formula \eqref{defc2}, but in the setup of this section,
  $f$ may be non-mixing and the correlations may not decay.
\end{rmk}

The ASIP for $S_n(\varphi)$ with variance $c^2$ was first proved in~\cite{MN05}.
Prior to our work, the best available rates were due to \cite{CM15,KKM},
formulated for the \emph{strong} polynomial moment of order $\beta$:
\[
  S_n(\varphi) - W_n
  =
  \begin{cases} 
    o\big(n^{1/\beta}(\log n)^{1/2}\big)  & \beta \in(2,4), \\
    O\big(n^{1/4}(\log n)^{1/2}(\log\log n)^{1/4}\big)  & \beta \ge 4.
  \end{cases}
\]

Again those rates are not better than $O(n^{1/4})$. Our main result is:
\begin{thm} \label{thm:NUE}
  Suppose that the return times of $f$ have a weak polynomial moment of order $\beta > 2$.
  Then $S_n(\varphi)$ satisfies the ASIP  with variance $c^2$ given by \eqref{eq:c2nomix}
  and rate $o(n^{ 1/\beta} (\log n)^{ 1/ \beta  + \eps} )$ for all $\eps >0$. 
  Further, if the return times of $f$ have a strong polynomial moment of order $\beta >2$,
  then the rate is $o(n^{1/\beta})$.
\end{thm}

In the remainder of this Section we prove Theorem~\ref{thm:NUE}.
First we consider the special case when $c^2=0$.

\begin{prop}
  \label{prop:cobas}
  Suppose that the return times of $f$ have a weak polynomial moment of order $\beta$.
  Then on the probability space $(Y, \mu_Y)$,
  \[
    \max_{k \leq n} \tau \circ F^k
    = o (n^{1/\beta} (\log n)^{1/\beta + \eps})
    \quad \text{almost surely for all } \eps > 0 \, .
  \]
  With a strong polynomial moment of order $\beta$,
  \[
    \max_{k \leq n} \tau \circ F^k
    = o (n^{1/\beta})
    \quad \text{almost surely.}
  \]
\end{prop}

\begin{proof}
  The sequence $(\tau \circ F^n)_{n \geq 0}$ is stationary
  and by the Borel-Cantelli lemma,
  it suffices to check that for all $\delta>0$, 
  \[
   \sum_{n\geq 1} \mu_Y \left ( \tau >  \delta n^{ 1/\beta} (\log n)^{ 1/ \beta  + \eps} \right )
   < \infty \, .
  \]
  Since $d \mu_Y / d\mu$ is bounded, it is enough to verify that
  \[
   \sum_{n\geq 1} m \left ( \tau >  \delta n^{ 1/\beta} (\log n)^{ 1/ \beta  + \eps} \right )
   < \infty \, ,
  \]
  which follows immediately from our assumptions.
  
  The proof for the strong polynomial moments is similar.
  (See also~\cite[Prop.~2.6]{KKM}.)
\end{proof}

\begin{cor}
  Theorem~\ref{thm:NUE} holds when $c^2 = 0$.
\end{cor}

\begin{proof}
  In~\cite{KKM}, the ASIP for nonuniformly expanding dynamical systems uses the 
  martingale-coboundary decomposition. With $c^2 = 0$, the martingale part 
  vanishes \cite[Cor.~2.12, Cor.~3.4]{KKM}. The estimates of the coboundary part
  are reduced to those in Proposition~\ref{prop:cobas}, see the proof of
  \cite[Prop.~2.6]{KKM}.
\end{proof}

\medskip

From here on, we assume that $c^2>0$.
We construct a Markov chain as in Section~\ref{sec:RMC}.
The general setup of nonuniformly expanding maps brings in
a few minor technical complications that we explain below:
\begin{itemize}
  \item Proposition~\ref{prop:nai} is the basis of the Markov chain construction,
    providing a ``regenerative'' decomposition of the reference measure.
    It is proved in the general setup in \cite{K17}.
  \item Construction of the semiconjugacy $\pi \colon \Omega \to X$ needs
    additional work, because we may not be able to define it everywhere as we did for the
    intermittent maps. Nevertheless, using assumption~\eqref{closure} we define it almost everywhere
    as follows.
    
    Let $\cA_n \subset \cA$ denote the set of all words with length $n+1$. Let
    \begin{align*}
      Z &= \cup_{n\in \N} \cup_{w\neq w'\in \cA_n}(\bar{Y}_w\cap \bar{Y}_{w'}) \, , \\
      \tilde{Y} &= (\cap_{n\in \N} \cup_{w\in \cA_n}Y_w) \setminus Z \, .
    \end{align*}
    Then $m(Z) = 0$ and $m(\tilde{Y}) = 1$, and 
    for every $y\in \tilde Y$ there exists a \emph{unique} sequence
    $(a_n)_{n\in \N} \in \alpha^{\N}$ such that $y\in \cap_{n\in \N}Y_{a_0 \cdots a_n}$.
    
    We endow $\alpha^{\N}$ with the metric 
    $\delta((a_n)_{n\in \N},(a_n')_{n\in \N})=\kappa^{-s}$,
    where $s \geq 0$ is the largest such that 
    $a_0 \cdots a_{s-1} = a'_0 \cdots a'_{s-1}$.
    Define a map $\chi \colon \alpha^{\N} \to Y$ by $\chi((a_n)_{n\in \N})=y$,
    where $\{y\}=\cap_{n\in \N}\bar Y_{w_n}$. It follows from completeness of $X$ and
    expansion of $F$ that $\chi$ is defined everywhere and is Lipschitz.
    Set $\cX:= \chi^{-1}(\tilde Y)$. Then $\cX$ is measurable and
    $\{\chi((a_n)_{n\in \N})\}=\cap_{n\in \N}Y_{a_0 \cdots a_n}$ for every $(a_n)_{n\in \N} \in \cX$.
  
    Every $g \in \Omega$ can be written as
    \[
      g = \bigl(
        (w_0, \ell_0), \ldots, (w_0, h(w_0)-1), (w_1,0), \ldots, (w_1, h(w_1)-1),
        (w_2, 0), \ldots
      \bigr)
      .
    \]
    Let \(\Omega_0 = \{g \in \Omega : \ell_0 = 0 \}\).    
    Define $\iota \colon \Omega_0 \to \alpha^\N$ by
    $\iota(g) = (a_0, a_1, \ldots)$ where $a_0 a_1 \cdots$ = $w_0 w_1 \cdots$.
    Define $\Omega'_0 =\iota^{-1} (\cX) $ and
    $\pi_0 \colon \Omega'_0 \to \tilde Y$, $\pi_0 = \chi \circ \iota$.
    Let
    \[
      \Omega' = \{ \sigma^\ell(g) : g \in \Omega'_0, \, 0 \leq \ell < h(w_0) \}
      \, .
    \]
    Observe that  $\bP_{\Omega} (\Omega') = 1$. Define a projection
    $\pi \colon \Omega' \to X$ by $\pi (g)= f^{\ell_0} (\pi_0 (g^{(0)}))$.
    Following the proof of Lemma~\ref{lem:proj} with straightforward changes, we see that
    $\pi$ is Lipschitz on $ \Omega'$.
\end{itemize}

\begin{rmk}
  Construction of the Markov chain for nonuniformly expanding dynamical
  systems can be found in~\cite{K17r}, done in different notation.
  There the space $X$ is not assumed complete, and 
  a more general, though less hands-on, assumption is used in place of \eqref{closure}:
  that the set
  \[
    \{ (a_0, a_1, \ldots) \in \alpha^\bN :
    \text{there exists }y \in Y \text{ with } F^k(y) \in a_k \text{ for all } k  \}
  \]
  is measurable in \(\alpha^{\bN}\) (in the product topology with Borel sigma algebra).
\end{rmk}

Further we work in notation of Section~\ref{sec:RMC}. Let
\[
  p = \gcd \{h(w) \colon w \in \cA\}
  .
\]
For the maps~\eqref{eq:LSV} and~\eqref{eq:HG} we showed that $p = 1$. This means that
the Markov chain $g_0, g_1, \ldots$ is aperiodic, which was necessary to control the
moments of the meeting time in Section~\ref{sec:MT}.
In the general case, however, it could be that $p \geq 2$. This is typical for example
for logistic maps with Collet-Eckmann parameters.

For $p = 1$, our proof proceeds without changes.
Below we treat the periodic case $p \geq 2$.  For $0 \leq k < p$, define
\[
  \tS_k = \{(w, \ell) \in S \colon \ell \equiv k \;(\bmod\; p)\}
\]
and
\[
  \Omega_k = \{(g_0, g_1, \ldots) \in \Omega \colon g_0 \in \tS_k\}
  .
\]
The sets $\Omega_k$ partition $\Omega$, and they are 
cyclically  permuted by $\sigma$: 
$\sigma(\Omega_k) = \Omega_{k+ 1 \;\bmod\; p}$.

Note that if $g_{n p} = (w, \ell) \in \tS_0$ for some $n \geq 0$,
then $g_{np+k} = (w, \ell + k)$ for $0 \leq k < p$.
Thus we can identify $\Omega_0$ with
\[
  \tOmega = \{(g_0, g_p, g_{2p} \ldots) \in \Omega \colon g_0 \in \tS_0\}
  .
\]

Let now $(\tg_0, \tg_1, \ldots)$ be a Markov chain with state space $\tS_0$ and transition probabilities
\begin{equation*}
  \begin{aligned}
  \bP (\tg_{n+1} = (w, \ell p) & \mid \tg_n = (w', \ell' p))
  \\ & = \begin{cases}
    1, & \ell = \ell' + 1 \text{ and } \ell' + 1 < h(w) /p  \text{ and } w=w' \\
    \bP_\cA(w) , & \ell = 0 \text{ and } \ell' + 1 = h(w') /p  \\
    0, & \text{else}
  \end{cases}
  \end{aligned}
\end{equation*}
This Markov chain admits a unique (ergodic) invariant probability measure $\tnu$ on $\tS_0$
given by
\beq \label{defnutilde}
\tnu ( w, \ell p) = p  \frac{ \bP_{\cA} (w) {\bf 1}_{\{0 \leq \ell  < h( w) /p  \}}}{   \bE_{\cA} (h)  }  \, .
\eeq
The Markov chain $(\tg_n)_{n \geq 0}$ starting from $\tnu$ defines a  probability measure $\bP_{\tOmega}$ on
the space \(\tOmega \). Note that $\bP_{\tOmega}$ corresponds to $\bP_\Omega$ conditioned on $\Omega_0$. Note also that  $\tg_0, \tg_1, \ldots$ is a Markov chain,
identical to $g_0, g_1, \ldots$ in structure except that
it is aperiodic and the return times
to \(S_0 = \{(w,\ell) \in S \colon \ell = 0\}\) are divided by $p$.

Following Section~\ref{sec:RMC}, we define the separation time $\ts$ and
the separation metric $\td$ on $\tOmega$, using the same constant $\lambda > 1$.
Suppose that $\ta,\tb \in \tOmega$
with the corresponding $a,b \in \Omega_0$.
The separation time is measured in terms of returns to $S_0$, hence
\[
  \ts(\ta, \tb) = s(a,b)
  \quad \text{and} \quad
  \td(\ta, \tb) = d(a,b)
\,   .
\]
Further, $d(\sigma^k (a), \sigma^k(b)) = d(a,b)$ for $0 \leq k < p$.
Let $\tpsi \colon \tOmega \to \bR$,
\[
  \tpsi(\ta) = \sum_{k=0}^{p-1} \psi(\sigma^k(a))
\,   .
\]
It follows that $\Lip \tpsi \leq p \Lip \psi$.
Also, $\tpsi$ is mean zero with respect to $\bP_\tOmega$. 

In Corollary \ref{cor:perper} of Appendix~\ref{sec:periodic} we show that the ASIP for
$\sum_{k=0}^{n-1} \psi \circ \sigma^k$ on $(\Omega, \bP_\Omega)$ (and hence the ASIP for $\sum_{k=0}^{n-1} \varphi \circ f^k$ on $(X, \mu)$) 
follows from the ASIP for
$\sum_{k=0}^{n-1} \tpsi \circ \tsigma^k$ on $(\tOmega, \bP_\tOmega)$
with the same rates and variance $v^2/p$, where $v^2$ is the
variance of the Wiener process on $({\tilde \Omega}, {\mathbb P}_{\tilde \Omega})$.

Now, as in Section~\ref{sec:proofASIP},  the  ASIP for
$\sum_{k=0}^{n-1} \tpsi \circ \tsigma^k$ on $(\tOmega, \bP_\tOmega)$ follows from the ASIP for $\sum_{k=0}^{n-1} {\tilde X}_k $ where, for any $k \geq 0$, 
\[
{\tilde X}_k = \tpsi \big (  (\tg_\ell)_{\ell \geq k}\big ) \, , 
\] 
and $(\tg_n)_{n \geq 0}$ is the stationary Markov chain defined above with the state space $\tS_0$ and stationary distribution $\tnu$. The proof of the ASIP for $\sum_{k=1}^n {\tilde X}_k $ with the adequate rates is, as in  Section~\ref{sec:proofASIP}, mainly based on suitable bounds for the tails of the meeting time ${\tilde T}$  for the Markov chain $(\tg_n)_{n \in \bN }$, which is defined as follows. First, without changing the distribution, we redefine $(\tg_n)_{ n \in \bN }$ on a new
probability space as follows. Let \(\tg_0 \in \tS_0\) be distributed according to \( \tnu\) (the stationary distribution defined by \eqref{defnutilde}).
Let \(\eps_1, \eps_2, \ldots\) be a sequence of independent
identically distributed random variables with values in \(\cA\),  distribution
\(\bP_\cA\) and independent from \(\tg_0\). For \(n \geq 0\), let
\[
  \tg_{n+1} 
  = {\tilde U}(\tg_n, \eps_{n+1}) \, , 
\]
where, for any $\ell \in \bN$, 
\beq \label{deftildeU}
  {\tilde U}((w, \ell p ), \eps)
  = \begin{cases}
    (w, (\ell+1)p), & \ell p < h(w) - p \, ,  \\
    (\eps, 0), & \ell p = h(w) - p \, .
  \end{cases}
\eeq
The meeting time ${\tilde T}$ of the Markov chain $(\tg_n)_{ n \in \bN }$ is then defined by 
\beq \label{defofTtilde}
  {\tilde T} = \inf \{n \geq 0 \colon \tg_n = {\tg_n}^* \} \, ,
\eeq
where $(\tg_n^*, n \in \bN )$ is the Markov chain defined as follows: \(\tg_0^*\) is a random variable in \(\tS_0\) with distribution $\tnu$ and 
independent from \( ( \tg_0, \{\eps_n\}_{n \geq 1} ) \) and, for $n \geq 0$,  $\tg_{n+1}^* = {\tilde U}(\tg_n^*, \eps_{n+1})$. Proceeding as in the proof of Lemma \ref{lem:nngh} and taking into account the  bounds on the tails of $h$  proved in \cite{K17}, we infer that the following lemma holds:

\begin{lemma} \label{lem:nnghtilde} \
\begin{itemize} 
  \item If the return times of $f$ have weak polynomial moment of order $\beta > 1$, then, for any $\eta >1$,  \(\bE (\tpsi_{\beta, \eta} ({\tilde T}) ) < \infty\), where $\tpsi_{\beta, \eta} (x) = x^{\beta-1} ( \log ( 1+x) )^{- \eta}$ for $x>0$. 
 \item If the return times of $f$ have strong polynomial moment of order $\beta > 1$, then  \(\bE({\tilde T}^{\beta-1} ) < \infty\). 
  \end{itemize}
\end{lemma}

In addition, proceeding as in the proof of  Lemma \ref{lmacovtildeX}, we also get the bound: 

\begin{lemma} \label{lmacovtildeXbis}  Assume that $\bE({\tilde T}) < \infty$. Then, for any $k \geq 1$ and any $\alpha \geq 1$, 
\[
|\Cov ({\tilde X}_0, {\tilde X}_k ) | \ll k^{- \alpha /2} + \bP ( {\tilde T} \geq [k / 4\alpha]) \, .
\]
\end{lemma} 

Now, with the same arguments as those developed in Section~\ref{sec:proofASIP} and taking into account
Lemmas \ref{lem:nnghtilde}  and 
\ref{lmacovtildeXbis},
we infer that, enlarging the  underlying probability space if necessary, there exists a sequence $(N_i)_{i \geq 1}$ of iid centered Gaussian r.v.'s with variance 
\beq \label{defseriesv2}
v^2 =  \Var ({\tilde X}_0) + 2 \sum_{k \geq 1} \Cov ({\tilde X}_0, {\tilde X}_k )
\eeq
such that,  for any $\kappa >1/\beta$, 
\begin{equation*} 
\sup_{k \leq n } \Big | \sum_{i=0}^{k-1} \tilde X_i - \sum_{i=1}^k N_i \Big  | = o( n^{ 1/\beta} (\log n)^{\eta \kappa} )  \quad  a.s. 
\end{equation*}
where $\eta =0$ if the return times of $f$ have strong polynomial moment of order $\beta > 2$ and $\eta =1$ if the return times of $f$ have weak polynomial moment of order $\beta > 2$. 

Now, according to  Corollary \ref{cor:perper},  the ASIP for $\sum_{k=0}^{n-1} \psi \circ \sigma^k$ on $(\Omega, \bP_\Omega)$ (and then the ASIP for $\sum_{k=0}^{n-1} \varphi \circ f^k$ on $(X, \mu)$)  holds with variance $v^2/p$ and 
rate $o( n^{ 1/\beta} (\log n)^{\eta \kappa} )$. It remains to  check that
$v^2/p=c^2$, with $c^2$ given by \eqref{eq:c2nomix}. Since by Lemmas \ref{lmacovtildeXbis} and \ref{lem:nnghtilde}, the series defined in \eqref{defseriesv2} is absolutely convergent, we have
\[
v^2 = \lim_{n \rightarrow \infty} \frac1 n  \Var \Big ( \sum_{k=0}^{n-1} \tilde X_n \Big )  = \lim_{n \rightarrow \infty} \frac 1 n \bE_{\bP_{\tilde \Omega}} \Big ( \Big ( \sum_{k=0}^{n-1} \tpsi \circ \tsigma^k  \Big )^2\Big ) \, .
\]
Hence, according to Lemma \ref{prop:ato}, 
\[
v^2  = \lim_{n \rightarrow \infty}  \frac 1 n
\bE_{\bP_{ \Omega}} \Big ( \Big ( \sum_{k=0}^{np-1} \psi \circ \sigma^k  \Big )^2\Big ) =  \lim_{n \rightarrow \infty} \frac 1 n 
\Vert S_{np} (\varphi) \Vert_{2, \mu}^2  = p c^2 \, .
\]
This ends the proof of Theorem~\ref{thm:NUE} when $c^2>0$.

\subsection{Optimality of the rates}

In this subsection we prove Proposition~\ref{lem:OR}. In fact we prove a stronger statement
as follows.
We consider a nonuniformly expanding map $f \colon X \to X$ as above.
We assume that:
\begin{itemize}
  \item $\tau$ is the first return time to $Y$;
  \item for some $\beta >2$, $\kappa > 0$ and all $n \geq 1$,
    \[
      m(\tau \geq n ) \geq \frac{\kappa}{n^{\beta}}
      \, .
    \]
\end{itemize}
These assumptions are verified for the map~\eqref{eq:LSV}
with $\beta = 1/\gamma$ and $Y = [1/2,1]$.

\begin{prop} 
  \label{prop-optimal}
Let $\psi$ be a bounded observable such that 
$\psi \equiv 0$ on $X \backslash Y$ and $\mu(\psi)>0$, 
and let $\varphi= \psi -\mu(\psi)$. 
Then for every process $(Z_n)_{n\in \bN}$ with the same law as $(\varphi\circ f^n)_{n\in \bN}$ and  every stationary and Gaussian centered sequence $(g_k)_{k \in {\mathbb Z}}$ such that $ n^{-1} {\rm Var }
\big (\sum_{i=1}^n g_i \big )$ converges, living on a same probability space,
\[
\limsup_{n \rightarrow \infty}  \, (n \log n)^{-1/\beta} \Big |  \sum_{k=1}^n Z_k -\sum_{k=1}^n g_k \Big | >0 \mbox{ almost surely.}
\]
\end{prop}

\begin{rmk}
  Under relaxed assumptions, there exist Lipshitz observables $\varphi$
  with $\int \varphi \, d\mu = 0$
  satisfying the hypotheses of Proposition~\ref{prop-optimal}.
  Indeed, if $m$ is regular and $\mu(\mathring{Y})\neq 0$, then 
  there is a compact $K\subset \mathring{Y}$ such that $\mu(K)>0$. 
  Note that $K$ and $\overline{X\backslash Y}$ are closed disjoint sets.
  Thus $\psi \colon X \to \bR$,
  \[
    \psi(x):= \frac{d(x,\overline{X\backslash Y})}{d(x,\overline{X\backslash Y})+d(x,K)} 
  \]
  is Lipschitz, and so is $\varphi = \psi - \int \psi \, d\mu$.
\end{rmk}

\begin{rmk}
  If $f \colon X \to X$ is a  Young tower~\cite{Y99}, then one can take 
  $\varphi = {\mathbf 1}_{Y} - \mu(Y)$. Then $\varphi$
  is Lipschitz with respect to the distance on the tower.
\end{rmk}

\begin{proof}[\it Proof of Proposition \ref{prop-optimal}]
Recall that $\mu_Y$ is the $F$-invariant probability measure on $Y$. 
For $n \geq 0$, let $\tau_n = \tau \circ F^n$.
We claim that $\mu_Y$-almost surely,
$n^{-1} \sum_{i=1}^n { \tau}_i  \to \int \tau \, d\mu$  as $n \rightarrow \infty$ 
and  ${ \tau}_n \geq  (n\log n)^{1/\beta}$ infinitely often. 
Then our result follows as in the proof of \cite[Prop.~15]{CDM17}.

It remains to verify the claim. Its first part is provided
by the pointwise ergodic theorem, so further we verify the second part.
We follow Gou\"ezel~\cite{GBC}. 

For $n\ge 0$, let $A_n = \{y \in Y \colon \tau_n(y) \geq (n \log n)^{1/\beta}\}$.
Recall that there is a constant $c > 0$ such that
for all $n,k \geq 0$,
\[
  \mu_Y (\tau_n = k) \geq c \, m(\tau = k)
  \, .
\]
Thus
\beq
  \label{eq:hhaah}
  \sum_{n=0}^\infty \mu_Y(A_n)
  \geq c \sum_{n=0}^\infty m\bigl(\tau \geq (n \log n)^{1/\beta}\bigr)
  = \infty
  \, .
\eeq
Next, there are constants $C > 0$ and $\theta \in ]0,1[$ such that
for all $k \neq n \geq 0$,
\[
  \bigl| \mu_Y(A_k \cap A_n)  - \mu_Y(A_k) \mu_Y(A_n)  \bigr| 
  \leq C \theta^{|n-k|} \mu_Y(A_k) \mu_Y(A_n) 
  \, .
\]
(See for instance the last line of~\cite[Sec.~1]{AD}.)
Therefore,
\begin{align*}
  \Bigl| \sum_{1 \leq k, \ell \leq n} & \mu_Y(A_k \cap A_\ell) 
  - \sum_{1 \leq k, \ell \leq n} \mu_Y(A_k) \mu_Y(A_\ell) \Bigr|
  \leq
  \sum_{1 \leq k, \ell \leq n} \bigl|\mu_Y(A_k \cap A_\ell) - \mu_Y(A_k) \mu_Y(A_\ell)\bigr|
  \\ & \ll \sum_{k=1}^{n} \mu_Y(A_k) +
  \sum_{1 \leq k, \ell \leq n} \theta^{|\ell-k|} \mu_Y(A_k) \mu_Y(A_\ell)
  \ll \sum_{k=1}^n \mu_Y(A_k) 
  \, .
\end{align*}
Taking into account~\eqref{eq:hhaah}, we obtain
\[
  \lim_{n \to \infty} 
  \frac{
    \sum_{1 \leq k, \ell \leq n} \mu_Y(A_k \cap A_\ell)
  }{
    \bigl( \sum_{k=1}^n \mu_Y(A_k) \bigr)^2
  }
  = 1 \, .
\]
By \cite[Lemma C]{ER}, we verify a criterion for the second Borel-Cantelli lemma
and prove that $\mu_Y( \cap_{n=1}^\infty \cup_{k=n}^\infty A_k) = 1$,
i.e.\ that $\mu_Y$-almost surely, $\tau_n \geq (n \log n)^{1/\beta}$ infinitely often.
This completes the proof of the claim.
\end{proof}

\appendix

\section{ASIP for periodic dynamical systems}
\label{sec:periodic}

Suppose that $(\Omega, \bP)$ is a probability space and
$\sigma \colon \Omega \to \Omega$ is a measure preserving transformation.

Suppose that $p \geq 2$ is an integer and $\sigma$ is $p$-periodic in
the sense that $\Omega$ can be partitioned into disjoint subsets
$\Omega_0, \ldots, \Omega_{p-1}$
which are permuted by $\sigma$ cyclically:
$\sigma(\Omega_k) = \Omega_{k+ 1 \;\bmod\; p}$. In particular $\bP ( \Omega_k ) =1/p $ for any $k=0, \dots, p-1$.  

Let $\tsigma \colon \Omega_0 \to \Omega_0$, $\tsigma = \sigma^p$.
We refer to $\tsigma$ as the \emph{induced map.}
The space $\Omega_0$ is endowed with a probability measure $\bP_0$, which is
$\bP$ conditioned on $\Omega_0$. Note that $\bP_0$ is invariant under $\tsigma$.

Suppose that $\psi \colon \Omega \to \bR$ is an observable
with $|\psi|_\infty = \sup_\Omega |\psi| < \infty$.
Define the \emph{induced observable} $\tpsi \colon \Omega_0 \to \bR$,
\[
  \tpsi(x) = \sum_{k=0}^{p-1} \psi(\sigma^k(x))
 \,  .
\]

Denote
\[
  \psi_n = \sum_{k=0}^{n-1} \psi \circ \sigma^k
  \quad \text{and} \quad
  \tpsi_n = \sum_{k=0}^{n-1} \tpsi \circ \tsigma^k
\,   .
\]
We consider $\psi_n$ and $\tpsi_n$ as random processes, defined
on probability spaces $(\Omega, \bP)$ and $(\Omega, \bP_0)$ respectively.  Define a projection $\pi_0 \colon \Omega \to \Omega_0$ by
  \begin{equation} \label{defpi0}
  \begin{aligned}
   \pi_0 (x)  & = \begin{cases}
    x &  \text{ if } x \in \Omega_0 \\
     \sigma^{p-k }(x)  & \text{ if } x \in \Omega_k \, , \, k=1, \ldots, p-1 \, .
       \end{cases}
  \end{aligned}
\end{equation}
\begin{lemma}
  \label{prop:ato} We have 
  \beq  \label{bounddiff}
    |\psi_n - \tpsi_{[ n/p]} \circ \pi_0|_\infty \leq 2 p |\psi|_\infty
  \,   .
\eeq
Moreover,  if $ \lim_{n \rightarrow \infty} n^{-1} \int_{\Omega}  \psi_n^2 \bP (\omega) \, d \omega = c^2$, then $ \lim_{n \rightarrow \infty} n^{-1} \int_{\Omega_0}  \tpsi_{[n/p]} ^2 \bP_0 (\omega) \, d \omega = c^2$. 
\end{lemma}
\begin{proof}
The bound  \eqref{bounddiff} is obvious. Indeed, for instance if $x \in \Omega_{1}$, it suffices to write
\[
|\psi_n - \tpsi_{[ n/p]} \circ \pi_0  | = \Big |  \sum_{k=0}^{n-1} \psi \circ \sigma^k -\sum_{k=p-1}^{ p {[ n/p ]}  +p -2} \psi \circ \sigma^k  \Big |  \leq 2 (p-1)
 |\psi|_\infty \, .
\]
To end the proof of the lemma, 
 note that $(\pi_0)_* \bP = \bP_0$,
  thus $\tpsi_n \circ \pi_0$, defined on the probability space $(\Omega, \bP)$, has
  the same distribution as $\tpsi_n$ on $(\Omega_0, \bP_0)$.
\end{proof}  
 \begin{cor}
  \label{cor:perper}
  Let $(b_n)_{n \geq 1}$ be a regularly varying sequence with values in ${\mathbb R}^+$, and such that $b_n (\log n)^{-1/2} \rightarrow \infty$ as 
  $n \rightarrow \infty$. 
 Assume that $\Omega$ can be enlarged in such a way that there exists
  a Brownian motion $\tW_t$ (with variance $v^2$) such that
  \[
    \tpsi_n \circ \pi_0 = \tW_n + o(b_n)
    \quad \text{almost surely.}
  \]
  Then, on the same probability space, there is a Brownian motion $W_t$ (with variance $c^2=v^2/p$) such that  
  \[
    \psi_n = W_n + o(b_n)
    \quad \text{almost surely.}
  \]
\end{cor}

\begin{proof}
By assumption and
  Lemma~\ref{prop:ato}, 
  \[
    \psi_n = \tW_{ [ n/p]} + o(b_n)
    \quad \text{almost surely.}
  \]
  Then $W_t = \tW_{t/p}$ is a Brownian motion (with variance $c^2=v^2/p$), and
  \[
    \sup_{s \leq t} |\tW_{s/p} - \tW_{[ s/p ]}| = O( (\log t )^{1/2})
    \quad \text{almost surely.}
  \]
(See, for instance, Theorem 3.2A in \cite{HR}). 
  The result follows.
\end{proof}

\subsection*{Acknowledgements} A.K.~is partially supported by 
a European Advanced Grant {\em StochExtHomog} (ERC AdG 320977);
is thankful to Ian Melbourne for helpful discussions;
acknowledges the warm hospitality of Paris-Est  and
Paris Descartes universities.

\end{document}